\documentclass[11pt,twoside]{amsart}
\usepackage[latin1]{inputenc}
\usepackage[T1]{fontenc}
\usepackage{mathtools}
\usepackage{graphicx}
\usepackage{tikz}
\usetikzlibrary{chains}

\usepackage[latin1]{inputenc}
\usepackage{amsmath}
\usepackage{amsthm}
\usepackage{amssymb}
\usepackage[all]{xy}
\usepackage{hyperref}
\newtheorem{thm}{Theorem}[section]

\newtheorem{prop}[thm]{Proposition}

\newtheorem{lem}[thm]{Lemma}
\newtheorem{cor}[thm]{Corollary}

\numberwithin{equation}{section}

\theoremstyle{definition}
\newtheorem{definition}[thm]{Definition}
\newtheorem{remark}[thm]{Remark}
\newtheorem{ex}[thm]{Example}

\newcommand{\Db}{{\rm D}^{\rm b}}

\newcommand{\Aut}{{\rm Aut}}

\newcommand{\Hom}{{\rm Hom}}

\newcommand{\Ext}{{\rm Ext}}

\newcommand{\cal}{\mathcal}
\newcommand{\ka}{{\cal A}}
\newcommand{\kb}{{\cal B}}
\newcommand{\kc}{{\cal C}}

\newcommand{\kh}{{\cal H}}

\newcommand{\ko}{{\cal O}}

\newcommand{\kt}{{\cal T}}

\newcommand{\ZZ}{\mathbb{Z}}

\newcommand{\RR}{\mathbb{R}}
\newcommand{\CC}{\mathbb{C}}

\newcommand{\HH}{{H\!H}}
\newcommand{\PP}{\mathbb{P}}

\newcommand{\OO}{{\rm O}}

\renewcommand{\to}{\xymatrix@1@=15pt{\ar[r]&}}
\renewcommand{\rightarrow}{\xymatrix@1@=15pt{\ar[r]&}}
\renewcommand{\mapsto}{\xymatrix@1@=15pt{\ar@{|->}[r]&}}
\renewcommand{\twoheadrightarrow}{\xymatrix@1@=18pt{\ar@{->>}[r]&}}
\renewcommand{\hookrightarrow}{\xymatrix@1@=15pt{\ar@{^(->}[r]&}}
\newcommand{\hook}{\xymatrix@1@=15pt{\ar@{^(->}[r]&}}
\newcommand{\congpf}{\xymatrix@1@=15pt{\ar[r]^-\sim&}}
\renewcommand{\cong}{\simeq}

\begin{document}

\title[]{Hochschild cohomology versus the Jacobian ring,\\
 and the  Torelli theorem for cubic fourfolds}

\author[D.\ Huybrechts and J.\ V.\ Rennemo]{{D.\ Huybrechts and J.\ V.\ Rennemo}}

\address{Mathematisches Institut,
Universit{\"a}t Bonn, Endenicher Allee 60, 53115 Bonn, Germany}
\email{huybrech@math.uni-bonn.de}

\address{Mathematical Institute, University of Oxford,
Woodstock Road, Oxford OX2 6GG, United Kingdom \&
All Souls College,  Oxford, OX1 4AL, United Kingdom }
\email{jvrennemo@gmail.com}
\begin{abstract} \noindent
The Jacobian ring $J(X)$ of a smooth hypersurface $X\subset\PP^{n+1}$ determines the isomorphism type
of $X$. This has been used by Donagi and others to prove the generic global Torelli theorem for hypersurfaces
in many cases.
However, in Voisin's  original proof (and, in fact, in all other proofs) of the global Torelli theorem for smooth cubic fourfolds 
$X\subset\PP^5$ the Jacobian ring does
not intervene. In this paper we present a proof of the global Torelli theorem for cubic fourfolds that relies
on the Jacobian ring and the (derived) global Torelli theorem for K3 surfaces. It emphasizes, once again, the close
and still mysterious relation between K3 surfaces and smooth cubic fourfolds.

More generally, for  a variant of Hochschild cohomology  $\HH^*(\ka_X,(1))$
of  Kuznetsov's cate\-gory $\ka_X$ (together with the degree shift functor $(1)$)
associated with an arbitrary smooth hypersurface $X\subset\PP^{n+1}$
of degree $d\leq n+2$ we construct a graded ring homomorphism $J(X)\twoheadrightarrow\HH^*(\ka_X,(1))$,
which is shown to be bijective whenever $\ka_X$ is a Calabi--Yau category.
 \vspace{-2mm}
\end{abstract}

\maketitle

{\let\thefootnote\relax\footnotetext{The first author was supported by the SFB/TR 45 `Periods,
Moduli Spaces and Arithmetic of Algebraic Varieties' of the DFG
(German Research Foundation).}
\marginpar{}
}

The derived category $\Db(X)$ of a smooth hypersurface $X\subset\PP^{n+1}$ of degree $d\ne n+2$
determines the hypersurface $X$ uniquely. However, a certain full triangulated subcategory
$\ka_X\subset\Db(X)$ introduced by Kuznetsov in \cite{Kuz0} turns out to be a subtler and more interesting derived invariant
of $X$. 
The case of cubic fourfolds $X\subset\PP^5$ 
has been studied intensively, cf.\ \cite{AT,HuyCubic,Kuz0,MacSte}. As observed by Kuznetsov, in this case
$\ka_X$ behaves in many respects like the derived category $\Db(S)$ of a K3 surface $S$. In
particular, its Hochschild cohomology $\HH^*(\ka_X)$ is known to be isomorphic
to the Hochschild cohomology of a K3 surface.

The aim of the paper is twofold. We introduce  a version of  Hochschild
cohomology of $\ka_X$  for smooth hypersurfaces $X\subset\PP^{n+1}$, denoted $\HH^*(\ka_X,(1))$, and
explain its relation to the Jacobian ring $J(X)$. The main result here is the following (cf.\ Corollary \ref{cor:Thm1}):

\begin{thm}\label{thm:main1} For any smooth hypersurface $X\subset\PP^{n+1}$ of degree $d\leq (n+2)/2$ there exists a
natural surjective homomorphism of graded rings $$\pi\colon J(X)\twoheadrightarrow\HH^*(\ka_X,(1)),$$
which is an isomorphism if $n+2$ is divisible by $d<n+2$.
\end{thm}

Note that the numerical assumption $d\mid (n+2)$ is exactly the one that according to Kuznetsov
\cite{KuzCY} ensures that $\ka_X$ is a
Calabi--Yau category. The theorem can be viewed as a graded version of Dyckerhoff's description of the Hochschild
cohomology of the category of (ungraded) matrix factorizations \cite{Dyck}.

As a consequence of Theorem \ref{thm:main1}, we provide a new proof of the global Torelli theorem for cubic fourfolds based on the (derived) global Torelli theorem for K3 surfaces.

\begin{thm}\label{thm:GT}
Two smooth complex cubic fourfolds $X,X'\subset\PP^5$  are isomorphic if and only if there exists a 
Hodge isometry $H^4(X,\ZZ)_{\rm pr}\cong H^4(X',\ZZ)_{\rm pr}$.
\end{thm}

The proof presented here passes via an isomorphism of Jacobian rings $J(X)\cong J(X')$ and so is closer in
spirit to Donagi's generic Torelli theorems for hypersurfaces \cite{Don}  than
to the original by Voisin \cite{VoisinGT}. We certainly make no claim that our arguments are
any easier or more natural than the existing ones. But it is certainly interesting to see
that the result can be deduced directly from the global Torelli theorem
for K3 surfaces,  demonstrating once more the fascinating and mysterious link between
cubic fourfolds and K3 surfaces.
\smallskip

In the rest of the introduction we  provide  more background for both parts of the paper.

\subsection{Global Torelli Theorem}
The classical Torelli theorem asserts that two smooth complex projective curves $C$ and $C'$ are isomorphic
if and only if there exists a Hodge isometry $H^1(C,\ZZ)\cong H^1(C',\ZZ)$.
The Hodge structure is the usual Hodge structure of weight one  and the pairing is provided by the intersection 
product.\footnote{If the condition on the compatibility with the pairing is dropped, then the Jacobians
of $C$ and $C'$ are still isomorphic, $J(C)\cong J(C')$, as unpolarized abelian varieties and,
in particular, $[S^nC]=[S^nC']$, for $n>2g-2$, in the Grothendieck ring of varieties $K_0({\rm Var})$.}

\smallskip Due to results of Pjatecki{\u\i}-{\v{S}}apiro and {\v{S}}afarevi{\v{c}} \cite{PSS} and Burns and Rapoport \cite{BR}, a similar result holds true
for K3 surfaces. More precisely,  two K3 surfaces $S$ and $S'$ are isomorphic if and only if there exists a
Hodge isometry $H^2(S,\ZZ)\cong H^2(S',\ZZ)$. The Hodge structure is of weight two and the pairing
is again given by the intersection product.\footnote{If the intersection product is ignored, then the Hodge
conjecture still predicts the isomorphism to be induced by an algebraic class on the product, but its concrete geometric meaning
is unclear.}

\smallskip More recently, Verbitsky \cite{Verb} proved a version of the global Torelli theorem for compact hyperk\"ahler
manifolds, higher-dimensional versions of K3 surfaces, which shows that two such manifolds
$Y$ and $Y'$ are birational if and only if there exists a Hodge isometry $H^2(Y,\ZZ)\cong H^2(Y',\ZZ)$
(with respect to the Beauville--Bogomolov pairing),
which is a parallel transport operator, cf.\ \cite{HuyGT,Mark}.

\smallskip

Classically a similar question has been asked for smooth hypersurfaces.
Concretely, are two smooth hypersurfaces $X,X'\subset\PP^{n+1}$ isomorphic if and only
if there exists a Hodge isometry between their primitive middle cohomology $H^n(X,\ZZ)_{\rm pr}
\cong H^n(X',\ZZ)_{\rm pr}$? The question has been addressed and answered in most cases
by Donagi \cite{Don}. Combined with the later work by Donagi--Green \cite{DonGr} and Cox--Green
\cite{CoxGr}, his results can be stated as follows: 

\begin{thm}[Donagi, Cox, Green]
The global Torelli theorem
holds for  generic hypersurfaces $X,X'\subset\PP^{n+1}$ of degree $d$ except possibly
in the following cases:
$${\rm (i)} ~(d,n)=(3,2),~ {\rm (ii)} ~(d,n)=(4,4m), \text{  and } {\rm (iii)} ~d\mid (n+2).$$
\end{thm}

Note that (iii) corresponds to the situation considered in Theorem \ref{thm:main1} and
ensures that $\ka_X$ is a Calabi--Yau category \cite{KuzCY}.

Also note that in the exceptions (i)-(iii), either $X$ is a Fano or a Calabi--Yau variety. The latter 
is the case
$d=n+2$ in (iii). In the Calabi--Yau situation,  the global Torelli theorem
is known to hold (and not only generically)
for the cases $(d,n)=(3,1)$ (elliptic curves),
$(d,n)=(4,2)$ (quartic K3 surfaces), and $(d,n)=(5,3)$ (quintic threefolds).
The first two are either trivial or special cases of the global Torelli theorem
for K3 surfaces. The case of quintic threefolds is much harder
and has been settled by Voisin in \cite{VoiQu}.

In the Fano situation, the global Torelli theorem really fails 
for $(d,n)=(3,2)$, for the Hodge structure of a cubic surface $X\subset\PP^3$ is of type $(1,1)$ and thus cannot
distinguish between non-isomorphic cubic surfaces.
So, the first interesting case is that of cubic fourfolds $X\subset \PP^5$.
Again, the global Torelli theorem is known to hold for those, a result due
to Voisin \cite{VoisinGT,VoisinErr}. Her proof eventually relies on the global Torelli theorem
for K3 surfaces of degree two (for which a direct proof was given by Shah \cite{Shah}).
Another proof for cubic fourfolds,  not drawing upon K3 surfaces, was given by Looijenga in \cite{Loo}
and yet another  more recent one by Charles \cite{Ch} uses Verbitsky's global Torelli theorem
applied to the hyperk\"ahler fourfold provided by the Fano variety of lines $F(X)$.

\smallskip

Donagi's proof of the generic global Torelli theorem for hypersurfaces uses the period map to
identify certain graded parts of the Jacobian rings of $X$ and $X'$. Applying his symmetrizer lemma \cite[Prop.\ 6.2]{Don}, for which one has
to exclude (i)-(iii), allows him to deduce from this a graded ring isomorphism $J(X)\cong J(X')$.
A version of the Mather--Yau theorem then implies $X\cong X'$. The argument breaks down for the exceptional
cases and, indeed, in the existing proofs of the global Torelli theorem for cubic fourfolds
the Jacobian ring makes no appearance.

\smallskip

The idea of our approach is to show that whenever there exists a
Hodge isometry $H^4(X,\ZZ)_{\rm pr}\cong H^4(X',\ZZ)_{\rm pr}$
between two (very general) smooth cubic fourfolds, then the K3 categories $\ka_X$ and $\ka_{X'}$ are equi\-valent.
This relies on the derived global Torelli theorem for K3 surfaces due to Orlov \cite{OrlovK3} and the 
result of Addington and Thomas \cite{AT} showing in particular that
the set of cubics $X$ for which $\ka_X$ is equivalent to the bounded derived category $\Db(S)$ of some K3 surface $S$ is dense. If an equivalence $\ka_X\cong\ka_{X'}$ in addition commutes with the natural
auto-equivalence $(1)$ given by mapping an object $E$ to the projection of $E\otimes \ko(1)$ (called the degree shift functor), then
$\HH^*(\ka_X,(1))\cong \HH^*(\ka_{X'},(1))$ essentially by definition of the Hochschild cohomology
of $(\ka,(1))$. Theorem \ref{thm:main1} then yields
a graded ring isomorphism $J(X)\cong J(X')$ and, by the Mather--Yau theorem, an isomorphism $X\cong X'$.
In order to reduce to the situation where the equivalence $\ka_X\cong\ka_{X'}$ indeed commutes with the degree shift functor,
one needs to argue that the set of cubics $X$ for which $\ka_X\cong\Db(S,\alpha)$ for some twisted K3 surface
$(S,\alpha)$ without any spherical objects is dense in the moduli space, cf.\ \cite{HuyCubic}. This suffices to conclude
the compatibility with the degree shift functor, as due to the results of \cite{HMS2} the group of auto-equivalences of $\Db(S,\alpha)$
is essentially trivial.

\subsection{Graded matrix factorizations}
Kuznetsov's category $\ka_X$ of a hypersurface $X\subset\PP^{n+1}$ defined
by an equation $f\in k[x_0,\ldots,x_{n+1}]$ has been shown to be equivalent to the category
of graded matrix factorizations ${\rm MF}(f,\ZZ)$, see \cite{Orlov} and Section \ref{sec:MF} for the
definition and some facts. Although we do not make use of this equivalence,
it served as a motivation for our approach. In particular, Dyckerhoff's description  \cite{Dyck} of the
Hochschild cohomology $\HH^*({\rm MF}(f))$ of the category of ungraded matrix factorization as the Jacobian
ring $J(X)$ got this project started.

More precisely, Dyckerhoff studies an isolated hypersurface singularity, i.e.\ a regular local $k$-algebra
$R$ and a non-unit $f\in R$ such that the quotient $R/(f)$ has an isolated singularity. He then shows that the Hochschild
cohomology $\HH^*({\rm MF}(f))$ (which is concentrated in even degree) of the dg-category of $\ZZ/2\ZZ$-periodic
matrix factorizations  is isomorphic to the Jacobian ring $R/(\partial_if)$, cf.\ \cite[Cor.\ 6.5]{Dyck}.

The naive original idea of our approach was to say that any equivalence $\ka_X\cong\ka_{X'}$, interpreted
as an equivalence ${\rm MF}(f,\ZZ)\cong{\rm MF}(f',\ZZ)$, that commutes with the degree shift functor $(1)$
on both sides, descends to an equivalence ${\rm MF}(f)\cong{\rm MF}(f,\ZZ)/(1)\cong{\rm MF}(f',\ZZ)/(1)\cong{\rm MF}(f')$.
The latter then induces a  ring isomorphism $J(X)\cong J(X')$.

\smallskip

There are, however, a number of problems that one has to address when using the equivalence $\ka_X\cong{\rm MF}(f,\ZZ)$.
First, this is an equivalence of triangulated categories. It comes with an enhancement, but in order to apply
any graded version of \cite{Dyck} one would need to make sure that the enhancement for ${\rm MF}(f,\ZZ)$
corresponds to the one used by Dyckerhoff.
Also, the compatibility of the equivalence $\ka_X\cong\ka_{X'}$ with the degree shift functor would need
to be lifted to the enhancement.
Second, the naive idea to pass from the category ${\rm MF}(f,\ZZ)$ to the quotient ${\rm MF}(f)={\rm MF}(f,\ZZ)/(1)$ needs
to be spelled out and possibly be lifted to the enhancements. Third, the relation between the degree
shift functors $(k)$ for ${\rm MF}(f,\ZZ)$ and the auto-equivalences of $\ka_X$ is rather technical, see \cite{BFK}.

So, we decided to work entirely on the derived side $\ka_X\subset\Db(X)$ and adapted Kuznetsov's philosophy that
viewing $\ka_X$ as an admissible subcategory of $\Db(X)$ and working exclusively with Fourier--Mukai
kernels replaces the choice of a dg-enhancement for $\ka_X$.

Versions of Hochschild cohomology for categories of graded
matrix factorizations have been introduced and studied in \cite{BFK2},
in which a relation to the Jacobian ring was also explained, see Section \ref{sec:MF}
for further comments.

\medskip

\noindent
{\bf Acknowledgments:} DH thanks Toby Dyckerhoff for explanations concerning \cite{Dyck}, Andrey Soldatenkov for his help with the literature
and Alex Perry for helpful discussions on Kuznetsov's work. JR 
thanks  Arend Bayer, John Calabrese for discussions and, especially, Matt Ballard, who originally suggested  to use \cite{Dyck} 
and the Mather--Yau theorem to recover a hypersurface from its category $\ka_X$.

\section{Kuznetsov's category $\ka_X$ from the kernel perspective}
In this section we first recall the definition of the category $\ka_X$ for a smooth hypersurface $X\subset\PP^{n+1}$ (over an arbitrary field
of characteristic zero)
of degree $d$ and state Kuznetsov's result saying that it
is a  Calabi--Yau category under suitable assumptions on $d$ and $n$. Then we revisit the auto-equivalence
$(1)\colon \ka_X\congpf\ka_X$, $E\mapsto i^*(E\otimes\ko_X(1))$ (the degree shift functor)
and Kuznetsov's
central observation that the $d$-fold composition $(d)$ is the double shift $[2]$. As for our
purposes it is important to understand this not only as an isomorphism of functors but as an isomorphism
between their Fourier--Mukai kernels, we essentially reprove his result in the kernel setting. This then allows
us to factor the isomorphism $(d)\cong[2]$ through the tangent bundle, see Lemma \ref{lem:alphabeta}, which is crucial for proving
the existence of the ring homomorphism from the Jacobian ring $J(X)$ to the Hochschild cohomology
$\HH^*(\ka_X,(1))$.

\subsection{} Let $X\subset \PP\coloneqq\PP^{n+1}$ be  a smooth hypersurface of degree $d$
and let $$\ka_X\coloneqq \langle\ko_X,\ldots,\ko_X(n+1-d)\rangle^\perp\subset\Db(X)$$
be the full triangulated subcategory of all objects $E$ with $\Hom^*(\ko_X(\ell),E)=0$ for all $\ell=0,\ldots,n+1-d$.
We will denote the image of $\ka_X$ under the line bundle twist $E\mapsto E\otimes\ko_X(\ell)$ by $\ka_X(\ell)$,
which can also be described as $\langle\ko_X(\ell),\ldots,\ko_X(n+1-d+\ell)\rangle^\perp$.

By definition, the left orthogonal ${}^\perp\ka_X$ is the full triangulated subcategory
spanned by the exceptional collection $\ko_X,\ldots,\ko_X(n+1-d)$. This yields a semi-orthogonal
decomposition  
 $$\Db(X)=\langle\ka_X,\ko_X,\ldots,\ko_X(n+1-d)\rangle,$$
 see \cite{BK,BO}.
In particular, the inclusion $i_*\colon \ka_X\,\hookrightarrow\Db(X)$
admits right and left adjoint functors $$i^!,i^*\colon \Db(X)\to \ka_X$$ (also called right and left projections),
so that there exist functorial isomorphisms $$\Hom_{\Db(X)}(i_*E,F)\cong \Hom_{\ka_X}(E,i^!F)\text{ and }\Hom_{\Db(X)}(F,i_*E)\cong
\Hom_{\ka_X}(i^*F,E)$$ for all $E\in\ka_X$ and $F\in\Db(X)$. With this notation, any object $E\in \Db(X)$ sits in  unique
exact triangles
$$E'\to E\to i_*i^*E~~Ê\text{~~ and~~ }~~i_*i^!E\to E\to E''$$
with $E'\in\langle\ko_X,\ldots,\ko_X(n+1-d)\rangle={}^\perp\ka_X$ and $E''\in\langle\ko_X(d-n-2),\ldots,\ko_X(-1)\rangle=\ka_X^\perp$.

The category $\Db(X)$ is endowed with a Serre functor described by $S_X\colon E\mapsto E\otimes\ko_X(d-n-2)[n]$
and a Serre functor on the admissible subcategory $\ka_X\subset\Db(X)$ is then given by
$S_{\ka_X}\cong i^!\circ S_X\circ i_*$. This isomorphism can also be read as a description of $i^!$ as $S_{\ka_X}\circ i^*\circ S_X^{-1}$.

The category $\ka_X$ is called a \emph{Calabi--Yau category} if $S_{\ka_X}$ is isomorphic to a shift functor $E\mapsto E[N]$,
in which case $N$ is called its dimension. For instance, if $d=n+2$, then $X$ itself is a Calabi--Yau variety
and, in particular, $\Db(X)$ is a Calabi--Yau category (of dimension $n$). 

See Remark \ref{rem:newpr} for an argument proving the next result in the case of cubic fourfolds.

\begin{thm}{\rm (Kuznetsov \cite[Thm.\ 3.5]{KuzCY})}\label{thm:KuzCY}
Assume $d\mid(n+2)$. Then $\ka_X$ is a Calabi--Yau category of dimension $(n+2)(d-2)/d$.
\end{thm}

Clearly, with $\ka_X$ also all twists $\ka_X(\ell)$ are Calabi--Yau categories (of the same dimension).

\begin{remark}\label{rem:noaccident}
As shall be explained, it is no accident that the Jacobian ring $$J(X)\coloneqq k[x_0,\ldots,x_{n+1}]/(\partial_i f)$$
of the equation $f\in k[x_0,\ldots,x_{n+1}]_d$ defining the hypersurface $X\subset\PP^{n+1}$
is of top degree $(n+2)(d-2)$, see \cite[Thm.\ 2.5]{Don}.
\end{remark}

\begin{ex}
The first interesting cases occur for $d=3$ and hypersurfaces $X\subset \PP^{n+1}$ of dimension $n=4,7,10,\ldots$. In these cases
the Calabi--Yau categories $\ka_X$ are of dimensions $N=2,3,4,\ldots$, respectively.
Besides the case of a quartic K3 surfaces $X\subset \PP^3$, in which case $\ka_X=\Db(X)$, the case
of a cubic fourfold $X\subset\PP^5$ is the only case that leads to a Calabi--Yau category of dimension two
(in this case a K3 category).
\end{ex}

\subsection{}  Kuznetsov \cite{KuzBas} associates with any subcategory $\kb\subset\Db(X)$, say full triangulated and closed under
taking direct summands,  a subcategory $\kb\boxtimes\Db(X)\subset\Db(X\times X)$.
By definition it is the smallest closed full triangulated subcategory closed under taking direct summands that contains
all objects of the form $E\boxtimes F\coloneqq p_1^*E\otimes p_2^*F$ with $E\in\kb$ and $F\in \Db(X)$. The subcategory $\Db(X)\boxtimes\kb$ is defined similarly.
Note that $\Db(X)\boxtimes\Db(X)\cong\Db(X\times X)$. According to \cite[Prop.\ 5.1]{KuzBas}, any semi-orthogonal decomposition
$\Db(X)=\langle\ka_1,\ldots,\ka_m\rangle$ induces a semi-orthogonal decomposition
$\Db(X\times X)=\langle\ka_1\boxtimes\Db(X),\ldots,\ka_m\boxtimes\Db(X)\rangle$.
Moreover, in this case (cf.\ \cite[Prop.\ 5.2]{KuzBas}):
\begin{equation}\label{eqn:Kuzaltdes}
\ka_i\boxtimes\Db(X)=\{E\in\Db(X\times X)\mid p_{1*}(E\otimes p_2^*F)\in\ka_i \text{ for all } F\in\Db(X)\}.
\end{equation}

We shall also need the exterior product $\kb\boxtimes\kb'\subset\Db(X\times X)$ of two categories $\kb,\kb'\subset\Db(X)$. As introduced in \cite[Sec.\ 5.5]{KuzBas},
this is the intersection of $\kb\boxtimes\Db(X),\Db(X)\boxtimes\kb'\subset\Db(X\times X)$.\footnote{The notation may suggest
to define $\kb\boxtimes\kb'$ as the smallest triangulated subcategory that is closed under taking direct summands and contains
all objects of the form $E\boxtimes E'$ with $E\in\kb$ and $E'\in\kb'$. But this \emph{a priori} produces a smaller
subcategory. However, if $\kb,\kb'$ are components of semi-orthogonal decompositions this description is valid, as was explained to us by Alex Perry.}  For two semi-orthogonal decompositions
$\Db(X)=\langle\ka_1,\ldots,\ka_m\rangle$ and $\Db(X)=\langle\ka_1',\ldots,\ka_n'\rangle$ the products
$\ka_i\boxtimes\ka_j'\subset\Db(X\times X)$ are admissible subcategories and, in fact, describe a semi-orthogonal decomposition
of $\Db(X\times X)$, see \cite[Thm.\ 5.8]{KuzBas}. The case of interest to us is the product
\begin{equation}\label{eqn:AA}
\ka_X(-(n+1-d))\boxtimes\ka_X\subset\Db(X\times X),
\end{equation}
which can alternatively be described as the subcategory right orthogonal to
$\langle\ko_X(-(n+1-d)),\ldots,\ko_X\rangle\boxtimes \Db(X)$ and $\Db(X)\boxtimes\langle\ko_X,\ldots,\ko_X(n+1-d)\rangle$.
We shall denote the inclusion (\ref{eqn:AA}) by $j_*$ and its right and left adjoint by 
\begin{equation}\label{eqn:jj}
j^!,j^*\colon \Db(X\times X)\to\ka_X(-(n+1-d))\boxtimes\ka_X\subset\Db(X\times X).
\end{equation}

As a consequence of Theorem \ref{thm:KuzCY} one finds

\begin{cor}\label{cor:KuzThm}
Assume $d\mid(n+2)$. Then for all $\ell$ the product $\ka_X(\ell)\boxtimes\ka_X$ is a Calabi--Yau category of dimension 
$2(n+2)(d-2)/d$.
\end{cor}

\begin{proof}
Consider the left and right projections $j^*,j^!\colon\Db(X\times X)\to\ka_X(-(n+1-d))\boxtimes\ka_X$ of the inclusion,
which can be written as the composition of left and right projections
${\rm id}\boxtimes i^*,{\rm id}\boxtimes i^!\colon\Db(X\times X)\to\Db(X)\boxtimes\ka_X$ and $i^*\boxtimes{\rm id},i^!\boxtimes{\rm id}\colon\Db(X)\boxtimes\ka_X\subset\Db(X\times X)\to \ka_X(-(n+1-d))\boxtimes\Db(X)$, for all of which the Fourier--Mukai
kernels (see below) are obtained by base change from the ones for $i^*$ and $i^!$.
Now, together with the comparison $S\circ j^*\cong j^!\circ S_{X\times X}$ of $j^*$ and $j^!$,  which
can also be read as a description of the Serre functor $S$ of $\ka_X(-(n+1-d))\boxtimes\ka_X$,
and the relation between $i^*$ and $i^!$ obtained from $S_{\ka_X}\cong[(n+2)(d-2)/d]$ (see Theorem \ref{thm:KuzCY}),
this yields the assertion.
\end{proof}
\subsection{} For $P\in \Db(X\times X)$ we denote by $\Phi_P\colon\Db(X)\to\Db(X)$ the Fourier--Mukai functor
$E\mapsto p_{2*}(p_1^*E\otimes P)$. Applying Kuznetsov's arguments \cite{KuzBas}, one easily finds

\begin{lem}\label{lem:Pcont}
{\rm (i)} The essential image of $\Phi_P$ is contained in $\ka_X\subset\Db(X)$ if and only if $P\in\Db(X)\boxtimes\ka_X$.

{\rm (ii)} There exists a factorization of $\Phi_P$ via the projection $i^*\colon \Db(X)\to \ka_X$, i.e.
$\Phi_P=0$ on $\langle\ko,\ldots,\ko(n+1-d)\rangle$, if and only if $P\in\ka_X(-(n+1-d))\boxtimes\Db(X)$.
\end{lem}

\begin{proof}
The first assertion follows from (\ref{eqn:Kuzaltdes}). For the second use that $\Phi_P(\ko_X(\ell))=0$
implies $\Hom_{X\times X}(p_1^*\ko_X(-\ell)\otimes p_2^*E,P)\cong\Hom_X(E,p_{2*}(p_1^*\ko_X(\ell)\otimes P))=0$
for all $E\in \Db(X)$. Hence, if $\Phi_P=0$ on $\langle\ko_X,\ldots,\ko_X(n+1-d)\rangle$, then $P\in\langle\ko_X(-(n+1-d)),\ldots,\ko_X\rangle^\perp\boxtimes\Db(X)=\ka_X(-(n+1-d))\boxtimes\Db(X)$.
\end{proof}

\begin{cor}
There exists a factorization
$$\xymatrix@R-7pt@C-6pt{\Db(X)\ar[dr]_{i^*}\ar[rrr]^-{\Phi_P}&&&\Db(X)\\
&\ka_X\ar[r]_{\bar\Phi_P}&\ka_X\ar[ur]_{i_*}&}$$
if and only if $P\in \ka_X(-(n+1-d))\boxtimes\ka_X$.\qed
\end{cor}

\begin{definition}\label{def:FM}
A functor $\Phi\colon \ka_X\to\ka_X$ is called a \emph{Fourier--Mukai functor} if it is isomorphic to
a functor of the form $\bar\Phi_P$ with $P\in  \ka_X(-(n+1-d))\boxtimes\ka_X$ as above.
\end{definition}

For any $P\in\Db(X\times X)$ one can consider $j^*P\in \ka_X(-(n+1-d))\boxtimes\ka_X$
which then defines a Fourier--Mukai functor $\bar\Phi_{j^*P}\colon\ka_X\to\ka_X$.

\begin{remark}\label{rem:FMequiCY}
The notion of a Fourier--Mukai functor $\bar \Phi_P\colon\ka_X\to\ka_{X'}$ between the Calabi--Yau categories of two different
hypersurfaces is defined similarly. In this case, the kernel is contained in $\ka_X(-(n+1-d))\boxtimes\ka_{X'}$.
\end{remark}

\begin{ex}\label{exa:id}
In the following, we shall denote by $\Delta=\Delta_X\subset X\times X$ the diagonal (and
also the diagonal embedding).

(i) Clearly, $\Phi_{\ko_{\Delta}}\cong{\rm id}$, but also $\bar\Phi_{j^*\ko_{\Delta}}\cong{\rm id}$.
We shall use the shorthand $$P_0\coloneqq j^*\ko_\Delta.$$

As observed in \cite[Prop.\ 3.8]{KuzHH} (see also Lemma \ref{lem:Pell}), $P_0$ can also be obtained as the image
of $\ko_\Delta$ under the left projection onto $\Db(X)\boxtimes\ka_X$ or onto $\ka_X(-(n+1-d))\boxtimes\Db(X)$. 
In other words, the cone of the adjunction $\ko_\Delta\to j^*\ko_\Delta\cong P_0$ is contained in the intersection of
$\langle\ko_X(-(n+1-d)),\ldots,\ko_X\rangle\boxtimes\Db(X)$ and $\Db(X)\boxtimes\langle\ko_X,\ldots,\ko_X(n+1-d)\rangle$. Indeed, the 
left projection image of $\ko_\Delta$ in $\Db(X)\boxtimes\ka_X$ considered as a Fourier--Mukai kernel
is automatically trivial on $\langle\ko_X(-(n+1-d)),\ldots,\ko_X\rangle$ and then use Lemma \ref{lem:Pcont}. Similarly, its image in $\ka_X(-(n+1-d))\boxtimes\Db(X)$ considered as a Fourier--Mukai kernel takes values in $\ka_X$ already.

(ii) The functor $\Phi_{\ko_\Delta(\ell)}\colon\Db(X)\to\Db(X)$ is the line bundle twist $E\mapsto E\otimes\ko_X(\ell)$.
We are particularly interested in the case $\ell=1$ and denote the projection of the corresponding kernel by
$$P_1\coloneqq j^*\ko_\Delta(1).$$ The induced functor was introduced by Kuznetsov
in \cite[Sec.\ 4]{Kuz1}:
$$(1)\coloneqq \bar\Phi_{P_1}\colon\ka_X\to\ka_X,~E\mapsto i^*(E\otimes\ko_X(1)).$$
\end{ex}

We call $(1)$ the \emph{degree shift functor}, which is motivated by interpreting $\ka_X$ as a category
of graded matrix factorizations, see Section \ref{sec:MF}.

\begin{remark}
Recall that for two objects $P,Q\in\Db(X\times X)$ the convolution $P\circ Q\in\Db(X\times X)$ is defined
as $p_{13*}(p_{12}^*P\otimes p_{23}^*Q)$. Then $\Phi_{P\circ Q}\cong\Phi_Q\circ\Phi_P\colon\Db(X)\to\Db(X)$,
see \cite[Ch.\ 5]{HuyFM}. Note that if $P,Q$ are contained in $\ka_X(-(n+1-d))\boxtimes\ka_X$, then so is
the convolution $P\circ Q$.

As an application, we rephrase the observation in Example \ref{exa:id}, (ii) and write $P_0$ as the convolution
\begin{equation}\label{eqn:P_0Conv}
P_0\cong[\ko(-m,m)\to\ko_\Delta]\circ[\ko(-m+1,m-1)\to\ko_\Delta]\circ\cdots\circ[\ko\to\ko_\Delta],
\end{equation}
where $\ko(a,b)\coloneqq\ko(a)\boxtimes\ko(b)$ and $m\coloneqq n+1-d$.
\end{remark}

\begin{remark}\label{rem:convP}
It is not difficult to show that for any $P\in \Db(X\times X)$ the projection $j^*P\in \ka_X(-(n+1-d))\boxtimes\ka_X$
is isomorphic to the right-left convolution with $P_0$, i.e.
$$j^*P\cong P_0\circ P\circ P_0.$$ Indeed, for every kernel $P$ there exists an exact triangle
$P'\to P\to j^*P$ with $P'$ contained in the category spanned by $\Db(X)\boxtimes\langle\ko_X,\ldots,\ko_X(n+1-d)\rangle$
and $\langle\ko_X(-(n+1-d)),\ldots,\ko_X\rangle\boxtimes\Db(X)$. Convoluting $P'$ with $P_0$ (or any object
in $\ka_X(-(n+1-d))\boxtimes\ka_X$) from both sides is trivial and, therefore, $P_0\circ P\circ P_0\cong P_0\circ j^*P\circ P_0$.
Similarly, using the arguments in Example \ref{exa:id}, (i), convoluting $\ko_\Delta'\to \ko_\Delta\to P_0$ with $j^*P$ from the left and with $j^*P\circ P_0$ from the right
yields isomorphisms $j^*P\cong j^*P\circ P_0\cong P_0\circ j^*P\circ P_0$.

As a special case, we record that
\begin{equation}\label{eqn:P1Conv}
P_1\cong P_0\circ\ko_\Delta(1)\circ P_0\cong P_0\circ\ko_\Delta(1)\circ[\ko\to\ko_\Delta],
\end{equation}
where for the second isomorphism we use (\ref{eqn:P_0Conv}) and $\ka_X(1)\subset\langle \ko_X(1),\ldots,\ko_X(m)\rangle^\perp$.
\end{remark}

The $\ell$-fold convolution of $P_1$ with itself yields $$P_\ell\coloneqq P_1^{\circ \ell}\coloneqq P_1\circ \ldots \circ P_1\in \ka_X(-(n+1-d))\boxtimes\ka_X,$$
whose induced Fourier--Mukai functor $\Phi_{P_\ell}$ is isomorphic to $$(\ell)\coloneqq(1)^\ell\colon\ka_X\to\ka_X.$$

\begin{remark}
In fact, Kuznetsov shows \cite[Cor.\ 3.18]{KuzCY} that the functor $(1)$ (which is his ${\rm O}|_{\ka_X}$)
is an equivalence and clearly so are all $(\ell)$. Alternatively, this can be deduced from Corollary \ref{cor:Pell}, see
Remark \ref{rem:reallyequ}.
\end{remark}

Note that in general neither is the kernel $P_\ell$  isomorphic to $j^*\ko_\Delta(\ell)$ nor is the functor $(\ell)$
isomorphic to $\bar\Phi_{j^*\ko_\Delta(\ell)}$. However, for $\ell=0,\ldots,d$ Kuznetsov establishes this isomorphism  \cite[Prop.\ 3.17]{KuzCY},
which is made explicit by the following kernel version, crucial for our purposes.
 See Corollary \ref{cor:Pell} for a characterization of all $P_\ell$.

\begin{lem}\label{lem:Pell}
Assume $d\leq (n+2)/2$. Then for all $\ell=0,\ldots,d$ there exist natural isomorphisms $$P_\ell\cong j^*\ko_\Delta(\ell).$$
\end{lem}

\begin{proof}
Consider the natural exact triangle $\ko_\Delta(\ell)'\to\ko_\Delta(\ell)\to j^*\ko_\Delta(\ell)$.
As in the proof of \cite[Prop.\ 3.8]{KuzHH} (cf.\ Example \ref{exa:id}, (i)), we shall show that $\ko_\Delta(\ell)'$, $\ell=0,\ldots,d-1$, is contained
in the subcategory spanned by $\langle\ko_X(-(n+1-d)),\ldots,\ko_X(-(n+1-d)+\ell-1)\rangle\boxtimes\Db(X)$ and $\Db(X)\boxtimes\langle\ko_X,\ldots,\ko_X(n+1-d)\rangle$.  For this consider the projection $k^*\colon\Db(X\times X)\to\langle\ko_X(-(n+1-d)),\ldots,\ko_X(-(n+1-d)+\ell-1)\rangle^\perp\boxtimes\ka_X$. It suffices to show that the kernel of $\ko_\Delta(\ell)\to k^*\ko_\Delta(\ell)$
is $\ko_\Delta(\ell)'$ or, equivalently, that $k^*\ko_\Delta(\ell)\cong j^*\ko_\Delta(\ell)$. 
Now, $\Hom^*(p_1^*\ko_X(a),k^*\ko_\Delta(\ell))=0$ holds for $a=-(n+1-d),\ldots,-(n+1-d)+\ell-1$ by definition of $k^*$
and for $a=-(n+1-d)+\ell,\ldots,0$, as $E\mapsto i^*(E\otimes\ko_X(\ell))$ is trivial on $\langle\ko_X,\ldots,\ko_X(n+1-d-\ell)\rangle$.


Next convolute the above exact triangle with $\ko_\Delta(1)$ from the left to obtain the exact triangle
$\ko_\Delta(1)\circ\ko_\Delta(\ell)'\to\ko_\Delta(\ell+1)\to\ko_\Delta(1)\circ j^*\ko_\Delta(\ell)$. As $\ko_\Delta(1)\circ\ko_\Delta(\ell)'$ is now contained in the span
of $\langle\ko_X(-(n+1-d)+1),\ldots,\ko_X(-(n+1-d)+\ell)\rangle\boxtimes\Db(X)$ and $\Db(X)\boxtimes\langle\ko_X,\ldots,\ko_X(n+1-d)\rangle$, it becomes trivial under $j^*$. This then implies that $j^*\ko_\Delta(\ell+1)\cong j^*(\ko_\Delta(1)\circ j^*\ko_\Delta(\ell))\cong P_1\circ P_\ell\cong P_{\ell+1}$ by induction. Note that the argument works as long as $-(n+1-d)+\ell\leq 0$ which under the assumption on $d$
holds for $\ell=0,\ldots,d-1$.
\end{proof}

\subsection{}\label{sec:Pell}
Consider a smooth hypersurface $X\subset \PP\coloneqq \PP^{n+1}$ of degree $d$. We write $\Delta_\PP\subset \PP\times\PP$
and $\Delta\coloneqq\Delta_X\subset X\times X$ for the two diagonals as well as for the corresponding closed immersions.

\begin{lem}
The pull-back of the structure sheaf $\ko_{\Delta_\PP}\in\Db(\PP\times\PP)$ under the
natural embedding $\varphi\colon X\times X\,\hookrightarrow \PP\times \PP$ is a complex
$\varphi^*\ko_{\Delta_\PP}\in \Db(X\times X)$ concentrated in degree $0$ and $-1$
with cohomology sheaves $\kh^0\cong \ko_{\Delta_X}$ and $\kh^{-1}\cong\ko_{\Delta_X}(-d)$.
\end{lem}

\begin{proof}
We view the cohomology sheaves $\kh^i$ of $\varphi^*\ko_{\Delta_\PP}$ as
$\kh^i\cong \kh^i(\ko_{X\times X}\otimes_{\PP\times \PP}\ko_{\Delta_\PP})$.
They are  supported on $\Delta_X$ and can be computed 
by means of the locally free resolution
$$E\coloneqq\left[\ko_{\PP\times\PP}(-d,-d)\to\ko_{\PP\times \PP}(-d,0)\oplus\ko_{\PP\times \PP}(0,-d)\to \ko_{\PP\times \PP}\right]\congpf
\ko_{X\times X}\in \Db(\PP\times\PP).$$
Hence, $\kh^i\cong \kh^i(E\otimes_{\PP\times \PP}\ko_{\Delta_\PP})\cong \kh^i[\ko_{\Delta_\PP}(-2d)\to\ko_{\Delta_\PP}(-d)^{\oplus 2}\to\ko_{\Delta_\PP}]$ and, in particular, $\kh^i=0$ for $i\ne 0,-1.-2$. Moreover, $\kh^{-2}\subset \ko_{\Delta_\PP}(-2d)$
as a sheaf supported on the proper subscheme $\Delta\subset\Delta_\PP$ has to be trivial, too. Obviously, $\kh^0\cong\ko_{\Delta_X}$
and, therefore $\kh^{-1}\cong {\rm Coker}(\ko_{\Delta_\PP}(-2d)\,\hookrightarrow\ko_{\Delta_\PP}(-d))\cong\ko_{\Delta_X}(-d)$.
\end{proof}

The usual exact triangle $\kh^{-1}[1]\to\varphi^*\ko_{\Delta_\PP}\to\kh^0$ for a complex concentrated in degree $0,-1$
twisted by $\ko(d,0)$ (or, equivalently, $\ko(0,d)$)
becomes
\begin{equation}\label{eqn:defalpha}
\xymatrix{\varphi^*\ko_{\Delta_\PP}(d)\ar[r]&\ko_{\Delta_X}(d)\ar[r]^-\alpha&\ko_{\Delta_X}[2].}
\end{equation}

The proof of the following result is close in spirit to Kuznetsov's arguments in the proof of \cite[Lem.\ 4.2]{Kuz1}.

\begin{lem}\label{lem:trivialvarphi}
Assume $d\leq(n+2)/2$. Then under the left adjoint $j^*\colon\Db(X\times X)\to \ka_X(-(n+1-d))\boxtimes\ka_X$ of the
natural inclusion (see (\ref{eqn:jj})), $\varphi^*\ko_{\Delta_\PP}(d)$ becomes trivial, i.e.
$$j^*\varphi^*\ko_{\Delta_\PP}(d)\cong 0.$$
\end{lem}

\begin{proof} The Koszul resolution
$$\left[\ko_\PP(-(n+1))\boxtimes\Omega_\PP^{n+1}(n+1)\to\ldots\to\ko_\PP(-1)\boxtimes\Omega_\PP(1)\to\ko_{\PP\times\PP}\right]\congpf\ko_{\Delta_\PP}$$
allows one to compute $\varphi^*\ko_{\Delta_\PP}(d)$ as
$[E_{n+1}\to\ldots\to E_0]$, where $E_i\coloneqq\ko_X(d-i)\boxtimes\Omega_\PP^i(i)|_X$.

Now, from the Euler
sequence $0\to\Omega_\PP(1)\to\ko_\PP^{\oplus n+2}\to \ko_\PP(1)\to 0$ and its alternating powers
$0\to \Omega_\PP^i(i)\to\ko_\PP^{\oplus {n+2\choose i}}\to\Omega_\PP^{i-1}(i)\to 0$ one deduces
$E_i\in\Db(X)\boxtimes \langle \ko_X,\ldots,\ko_X(n+1-d)\rangle$ for $i=0,\ldots,n+1-d$. As
$d\leq (n+2)/2$, the remaining $E_i$, $i=n+2-d,\ldots, n+1$ are all contained
in $\langle\ko_X(-(n+1-d)),\ldots,\ko_X\rangle\boxtimes\Db(X)$. Altogether, this shows
that $\varphi^*\ko_{\Delta_\PP}(d)$ is contained in the left orthogonal of $\ka_X(-(n+1-d))\boxtimes\ka_X$ and hence $j^*\varphi^*\ko_{\Delta_\PP}(d)\cong0$.
\end{proof}
Combining this with Lemma \ref{lem:Pell} yields the next result, which again is just the kernel
version of a result of Kuznetsov \cite{Kuz1,KuzCY}.

\begin{cor}\label{cor:Pell}
For $d\leq(n+2)/2$ one has $$P_d\cong j^*\ko_{\Delta_X}(d)\cong j^*\ko_{\Delta_X}[2] \text{ and }
(d)\cong [2]$$
and, more generally, $$P_{d\ell+a}\cong j^*\ko_\Delta(a)[2\ell]\text{ and } (d\ell+a)\cong(a)\circ[2\ell]$$
for all $\ell\geq0$ and $0\leq a<d$.\qed
\end{cor}

\begin{remark}\label{rem:reallyequ}
Note that this, \emph{a posteriori}, shows that $(1)\colon\ka_X\to\ka_X$ is indeed an auto-equivalence.
\end{remark}

\begin{remark}\label{rem:newpr} 
For the reader's convenience we briefly mention that at least for cubic fourfolds $X\subset\PP^5$ the arguments presented
so far already ensure that $\ka_X$ is a two-dimensional Calabi--Yau category. Indeed, in this case Lemma \ref{lem:Pell}
applies to $\ell=n+2-d$ and hence  for all $E,F\in\ka_X$ one finds $\Hom_{\ka_X}(i^*(E\otimes\ko_X(n+2-d)),F)\cong\Hom_{\Db(X)}(E\otimes\ko_X(n+2-d),F)\cong \Hom_{\Db(X)}(F,E[n])^*$, which can be read
as $S_{\ka_X}\circ(n+2-d)\cong[n]$. This together with $(d)\cong[2]$ yields the assertion.
\end{remark}
\subsection{}\label{sec:Atiyah}
We shall need an alternative description of the isomorphism $j^*\ko_\Delta(d)\cong j^*\ko_\Delta[2]$ that involves
the tangent bundle $\kt_X$.

The normal bundle sequence $0\to\kt_X\to\kt_\PP|_X\to\ko_X(d)\to 0$ is encoded by the
boundary map $\ko_X(d)\to\kt_X[1]$, which is certainly non-trivial for $d>2$ and $n>2$, which covers all cases
of interest to us. Taking direct
images under the diagonal morphism yields
\begin{equation}\label{eqn:boundary}
\ko_\Delta(d)\to\Delta_*\kt_X[1].
\end{equation}

The boundary morphism of the short exact sequence
$0\to\Delta_*\Omega_X\to\ko_{X\times X}/{\mathcal I}_\Delta^2\to\ko_\Delta\to 0$
is the universal Atiyah class ${\rm At}\colon \ko_\Delta\to \Delta_*\Omega_X[1]$.
Taking exterior powers, it yields a natural  map $\bigoplus {\rm At}^p\colon\ko_\Delta\to\Delta_*\bigoplus_{p=0}^n\Omega^p_X[p]$,
whose adjoint $$\Delta^*\ko_\Delta\congpf\bigoplus_{p=0}^n\Omega^p_X[p]$$ is known
to be an isomorphism \cite{Cal,Marka}. We shall rather work in the dual setting \cite[Sec.\ 8]{KuzHH}:
$$\bigoplus_{p=0}^n{\bigwedge}^p\kt_X[-p]\congpf\Delta^!\ko_\Delta.$$
Taking direct image under the diagonal and composing with the natural inclusion of $\kt_X[-1]$ on the left and
with the adjunction $\Delta_*\Delta^!\to{\rm id}$ on the right yields
\begin{equation}\label{eqn:Atcomp}
\Delta_*\kt_X[1]\to\Delta_*\bigoplus_{p=0}^n{\bigwedge}^p\kt_X[2-p]\congpf\Delta_*\Delta^!\ko_\Delta[2]\to\ko_\Delta[2].
\end{equation}

Now, composing (\ref{eqn:boundary}) and (\ref{eqn:Atcomp}) yields
a map
$$\beta\colon \ko_\Delta(d)\to \Delta_*\kt_X[1]\to\ko_\Delta[2],$$
which can be compared to $\alpha$ in (\ref{eqn:defalpha}).

\begin{lem}\label{lem:alphabeta}
The two maps $\alpha,\beta\colon \ko_\Delta(d)\to \Delta_*\kt_X[1]\to\ko_\Delta[2]$ coincide (up to non-trivial scaling).
\end{lem}

\begin{proof}
This can be seen as a consequence of \cite[Thm.\ 2.10]{HT}. Indeed, $\beta$ is by construction the composition of
the universal Atiyah class with the Kodaira--Spencer class for the embedding $X\subset\PP^{n+1}$ which coincides with the universal obstruction class $\alpha$.

Alternatively, one can show that $\Hom_{X\times X}(\ko_\Delta(d),\ko_\Delta[2])$ is just one-dimensional and that
both maps $\alpha$ and $\beta$ are non-zero. Indeed,  $\Hom_{X\times X}(\ko_\Delta(d),\ko_\Delta[2])\cong
H^0(X,\Delta^!\ko_\Delta(-d)[2])\cong H^0(X,\bigoplus_{p=0}^n\bigwedge^p\kt_X(-d)[2-p])$. For degree reasons, only
$H^2(X,\ko(-d))$, $H^1(X,\kt_X(-d))$, and $H^0(X,\bigwedge^2\kt_X(-d))$ contribute to the direct sum. The first and third
cohomology groups are obviously trivial and due to
the normal bundle sequence the second one is one-dimensional. 
Clearly, $\alpha\ne0$, as $j^*\alpha$ is an isomorphism by Lemma \ref{lem:trivialvarphi}.
Similarly one checks that  $\beta\ne0$, as otherwise the adjunction $\Delta^!\Delta_*\ko_\Delta\to\ko_\Delta$ would be zero.
Hence, $\alpha$ and $\beta$ differ at most by a non-trivial scalar.
\end{proof}

\begin{remark}
The interpretation of $\alpha$ as the universal obstruction class has the following geometric consequence
for objects in $\ka_X$: Non-trivial objects $E\in\ka_X$ are maximally  obstructed, i.e.\ they do not even deform to first order to the ambient
projective space. Indeed, $\alpha$ as a morphism between Fourier--Mukai kernels applied to any object $E\in\Db(X)$
yields the  obstruction  $o(E)\in\Ext^2(E\otimes\ko(d), E)$ to extend $E$ to the first order neighbourhood
of $X$ in $\PP$, cf.\ \cite{HT}. However, for $E\in \ka_X$ 
this class, via adjunction, yields the isomorphism $E(d)\cong i^*(E\otimes\ko(d))\congpf E[2]$ and so $o(E)\ne0$ for all
non-trivial $E\in\ka_X$.
\end{remark}

Composing $\ko_X^{\oplus n+2}\twoheadrightarrow\kt_\PP(-1)|_X$, coming from the restriction of the Euler
sequence, with the natural projection $\kt_\PP(-1)|_X\twoheadrightarrow\ko_X(d-1)$ in the normal bundle
sequence yields a map
\begin{equation}\label{eqn:gamma}
\xymatrix{\gamma\colon\ko_X^{\oplus n+2}\ar[r]&\kt_\PP(-1)|_X\ar[r]&\ko_X(d-1).}
\end{equation}
This map is induced by the partial derivatives $\partial_if\in H^0(X,\ko_X(d-1))$ of the equation $f\in k[x_0,\ldots,x_{n+1}]_d$
defining $X\subset\PP^{n+1}$.

\begin{cor}\label{cor:partialzero}
The induced map $j^*\Delta_*\gamma$ in the product category
$\ka_X(-(n+1-d))\boxtimes\ka_X\subset\Db(X\times X)$ is trivial:
$$0=j^*\Delta_*\gamma\colon j^*\ko_\Delta^{\oplus n+2}\to j^*\ko_\Delta(d-1).$$
\end{cor}

\begin{proof} First note that the composition $\alpha(-1)\circ \Delta_*\gamma\colon\ko_\Delta^{\oplus n+2}\to\ko_\Delta(d-1)\to\ko_\Delta(-1)[2]$ is trivial. Indeed, as $\alpha=\beta$ and $\beta$ factors through the boundary map $\ko_X(d-1)\to\kt_X(-1)[1]$ of the normal bundle
sequence,
this follows from the observation that already the composition $\kt_\PP(-1)|_X\to\ko_X(d-1)\to\kt_X(-1)[1]$ is trivial. Hence, $\Delta_*\gamma$
factors through a morphism $\delta\colon\ko_\Delta^{\oplus n+2}\to\varphi^*\ko_{\Delta_\PP}(d-1)$ 
(use the triangle (\ref{eqn:defalpha}) tensored by by $\ko_X(-1)$) and it suffices to show that $j^*\delta=0$.

Now, using the notation from the proof of Lemma \ref{lem:trivialvarphi}, we let $E_i'\coloneqq E_i\otimes\ko(-1,0)$. Then we still
have $E_i'\in \Db(X)\boxtimes\langle\ko_X,\ldots,\ko_X(n+1-d)\rangle$ for $i=0,\ldots,n+1-d$ and $E_i'\in\langle\ko_X(-(n+1-d)),\ldots,\ko_X\rangle\boxtimes\Db(X)$ for $i=n+2-d,\ldots, n$. Only the last one $E_{n+1}'\cong\ko(d-1-(n+1),-1)$ does not vanish under $j^*$. Hence, the pull-back
under $j^*$ of the natural
map $\varphi^*\ko_{\Delta_\PP}(d-1)\to E_{n+1}'[n+1]$ becomes an isomorphism $j^*\varphi^*\ko_{\Delta_\PP}(d-1)\congpf j^* E_{n+1}'[n+1]$.
Therefore, in order to prove the assertion, it suffices to show that the composition $\ko_\Delta^{\oplus n+2}\to\varphi^*\ko_{\Delta_\PP}(d-1)\to E_{n+1}'[n+1]$ is trivial, which follows from $\Ext^{n+1}_{X\times X}(\ko_\Delta,E_{n+1}')\cong H^{n-1}(X,\ko_X(d-n-1))^*=0$ using
Serre duality.
\end{proof}

 See Section \ref{sec:MF} for an interpretation of his result using the category of graded matrix factorizations.

\section{Extended Hochschild cohomology of hypersurfaces}
We define the Hochschild cohomology $\HH^*(\ka_X,(1))$ of the category $\ka_X$ endowed
with the degree shift functor $(1)$ associated with any
smooth hypersurface $X\subset\PP^{n+1}$. It is intimately related to the usual even Hochschild cohomology $\HH^*(\ka_X)$, but
incorporates also the degree shift functor. This section also contains the comparison of the Jacobian ring $J(X)$ with $\HH^*(\ka_X,(1))$ (cf.\ Theorem \ref{thm:main1}). We work again over an arbitrary field of characteristic zero.

\subsection{} As before, we let $X\subset \PP=\PP^{n+1}$ be a smooth hypersurface defined by a homogeneous
polynomial $f\in k[x_0,\ldots,x_{n+1}]_d$ of degree $d$. 
We set $\sigma\coloneqq (n+2)(d-2)$ and define $$L(X)\coloneqq \bigoplus_{\ell=0}^\sigma L_\ell(X)=\bigoplus_{\ell=0}^\sigma\Hom(P_0,P_\ell)$$
(with $\Hom$ taken in the full subcategory $\ka_X(-(n+1-d))\boxtimes\ka_X\subset\Db(X\times X)$)
with its natural ring structure defined by composition. More precisely, by
applying convolution with $P_\ell$ one obtains a natural map
$\Hom(P_0,P_{\ell'})\to\Hom(P_\ell\circ P_0,P_{\ell}\circ P_{\ell'})\cong\Hom(P_\ell,P_{\ell+\ell'})$,
which then yields
$$\Hom(P_0,P_\ell)\times\Hom(P_0,P_{\ell'})\to\Hom(P_0,P_\ell)\times\Hom(P_\ell,P_{\ell+\ell'})\to
\Hom(P_0,P_{\ell+\ell'})$$
with the convention that the multiplication is trivial as soon as $\ell+\ell'$ exceeds $\sigma$.
Standard arguments show that this endows $L(X)$ with the structure of a graded commutative
ring.

Next, consider the  natural map
\begin{equation}\label{eqn:ringone}\small
H^0(\PP,\ko_\PP(1))\cong H^0(X,\ko_X(1))\congpf\Hom(\ko_\Delta,\ko_\Delta(1))\to\Hom (j^*\ko_\Delta, j^*\ko_\Delta(1))=L_1(X).\end{equation}

\begin{lem}\label{lem:O(1)}
The homomorphism (\ref{eqn:ringone}) is surjective.
\end{lem}

\begin{proof} (i)  We shall first show that (\ref{eqn:ringone}) is bijective for $d\geq 4$ and injective for $d\geq 3$.
Consider the exact triangle $P_1'\to \ko_\Delta(1)\to P_1$. Using (\ref{eqn:P1Conv}), a direct computation
shows that
$$P_1\cong\left[H^0(X,\ko_X(1))^{\otimes m+1}\otimes\ko(-m,0)\to\cdots\to\bigoplus_{i=-1}^m\ko(-i,i+1)\to\ko_\Delta(1)\right],$$
where one uses that all pull-backs, direct images and tensor products are in fact underived due to $H^{>0}(X,\ko(i))=0$.
Hence,
$$\Hom(\ko_\Delta,P_1'[i])\cong\Hom(\ko_X,\Delta^*P_1'\otimes\omega_X^*[i-n])=0$$
for $i-n<-m-1$ or, equivalently, for $i<d-2$. Therefore, $\Hom(\ko_\Delta,P_1')=0=\Hom(\ko_\Delta,P_1'[1])$
for $d\geq 4$, which proves the bijectivity of  (\ref{eqn:ringone}) in these cases. For $d=3$ we
still have $\Hom(\ko_\Delta,P_1')=0$, which proves at least the injectivity.\footnote{Note that for $d\mid (n+2)$ and $d<n+2$, the injectivity also follows from Proposition \ref{prop:JacHH2}.}

(ii)  Surjectivity also holds for $d=3$, but since the calculation is
long and we do not have any applications for $n>4$, we only give the proof in
the case $n=4$.

By definition and Serre duality for the Calabi--Yau category $\ka_X(-2)\boxtimes\ka_X$ (which is of dimension four),
$L_1(X)=\Hom(j^*\ko_\Delta,j^*\ko_\Delta(1))\cong\Ext^4(j^*\ko_\Delta(1),j^*\ko_\Delta)^*\cong \Ext^4(\ko_\Delta(1),j^*\ko_\Delta)^*
\cong H^4(X,\Delta^!j^*\ko_\Delta\otimes\ko_X(-1))^*\cong H^0(X,\Delta^*j^*\ko_\Delta\otimes\ko_X(2))^*$,
where we use that $\Delta^!\cong(\omega_X^*\otimes(~))\circ\Delta^*\circ [-4]$.

In order to compute $\Delta^*j^*\ko_\Delta$, consider the left projections
$j_0^*\colon\Db(X\times X)\to\Db(X)\boxtimes\langle\ko_X(2)\rangle^\perp$
and $j_1^*\colon\Db(X\times X)\to\Db(X)\boxtimes\langle\ko_X(1),\ko_X(2)\rangle^\perp$.
Then $j^*\ko_\Delta$ is recursively described by the exact triangles
\begin{equation}\label{eqn:jDelta}
\begin{array}{rcl}
\ko_\Delta\circ\ko(-2,2)\to&\ko_\Delta&\to j_0^*\ko_\Delta,\\
j_0^*\ko_\Delta\circ\ko(-1,1)\to &j_0^*\ko_\Delta&\to j_1^*\ko_\Delta,\\
j_1^*\ko_\Delta\circ\ko\to &j_1^*\ko_\Delta&\to j^*\ko_\Delta.
\end{array}
\end{equation}
Applying $\Delta^*$ yields the diagram of exact triangles
$$\xymatrix{\Delta^*j_0^*\ko_\Delta\ar[d]&&\Delta^*(j_0^*\circ\ko)[1]\ar[d]\\
\Delta^*j_1^*\ko_\Delta\ar[d]\ar[r]&\Delta^*j^*\ko_\Delta\ar[r]&\Delta^*(j_1^*\ko_\Delta\circ\ko)[1]\ar[d]\\
\Delta^*(j_0^*\ko_\Delta\circ\ko(-1,1))[1]&&\Delta^*(j_0^*\ko_\Delta\circ\ko(-1,1)\circ\ko)[2].}$$
The computation of the four corners is straightforward. For example, the lower-left corner is $\Delta^*(j_0^*\ko_\Delta\otimes\ko(-1,1))\cong \Omega_\PP|_X[1]$,
as it is the cone of the natural map
$$H^0(X,\ko_X(1))\otimes\ko_X(-1)\cong\Delta^*(\ko_\Delta\circ \ko(-2,2)\circ\ko(-1,1))\to \Delta^*(\ko_\Delta\circ\ko(-1,1))\cong\ko_X.$$
Similarly, the lower-right corner is $\Delta^*(j_0^*\ko_\Delta\circ\ko(-1,1)\circ\ko)[2]\cong(\Omega_\PP|_X(-1)\otimes H^0(X,\ko_X(1)))[3]$.
Indeed, it is the $[2]$-shift of the cone of the map
$$H^0(X,\ko_X(1))^{\otimes 2}\otimes\ko_X(-2)\cong\Delta^*(\ko_\Delta\circ\ko(-2,2)\circ\ko(-1,1)\circ\ko)\to\Delta^*(\ko_\Delta\circ\ko(-1,1)\circ\ko),$$
which is the evaluation $H^0(X,\ko_X(1))\otimes\ko_X(-2)\to\ko_X(-1)$  tensored with $H^0(X,\ko_X(1))$.
For the computation of the upper-left corner use $\Delta^*\ko_\Delta\cong\bigoplus\Omega_X^p[p]$ and $\Delta^*(\ko_\Delta\circ\ko(-2,2))\cong\ko_X$
to conclude that $\Delta^*(j_0^*\ko_\Delta)\cong\bigoplus_{p\geq1}\Omega_X^p[p]$. Finally, the upper-right corner
is the $[1]$-shift  of the cone of
$$H^0(X,\ko_X(2))\otimes\ko_X(-2)\cong\Delta^*(\ko_\Delta\circ\ko(-2,2)\circ\ko)\to\Delta^*(\ko_\Delta\circ\ko)\cong\ko_X$$
and so $\Delta^*(j_0^*\circ\ko)[1]\cong\Omega_{\PP'}|_X[2]$, where $X\subset \PP'=\PP^{20}$ is the second Veronese embedding.

Now, in order to compute $H^0(X,\Delta^*j^*\ko_\Delta\otimes\ko_X(2))$, use the induced diagram of triangles
$$\xymatrix{\bigoplus_{p\geq1}H^p(\Omega_X^p\otimes\ko_X(2))\ar[d]&&H^2(\Omega_{\PP'}|_X\otimes\ko_X(2))\ar[d]\\
H^0(\Delta^*j_1^*\ko_\Delta\otimes\ko_X(2))\ar[d]\ar[r]&H^0(\Delta^*j^*\ko_\Delta\otimes\ko_X(2))\ar[r]&H^1(\Delta^*(j_1^*\ko_\Delta\circ\ko)\otimes\ko_X(2))\ar[d]\\
H^2(\Omega_\PP|_X\otimes\ko_X(2))&&H^3(\Omega_\PP|_X(1)).}$$
Using the Euler and conormal bundle sequences (and their exterior powers) and standard vanishing results, one easily shows that the only non-trivial term in the four corners
is the term $H^2(X,\Omega_X^2\otimes\ko_X(2))\cong H^3(X,\Omega_X(-1))\cong H^4(X,\ko_X(-4))\cong H^0(X,\ko_X(1))^*$
in the left-upper corner. 

This shows that  in (\ref{eqn:ringone}) one has $ \dim H^0(X,\ko_X(1))\geq \dim  L_1(X)$
and together with the injectivity from (i) this proves the surjectivity.
\end{proof}

\begin{remark}
Note that the analogous maps $H^0(X,\ko(\ell))\to L_\ell(X)$ for arbitrary $\ell$ are usually
not surjective. For example, in the case of the cubic fourfold $L_3(X)$ is of dimension $22$, but
by Proposition \ref{prop:JacHH}
$H^0(X,\ko_X(3))\to L_3(X)$ factors through $J_3(X)\cong H^1(X,\kt_X)$ (see below) which is of dimension $20$.
\end{remark}

Next, consider the homomorphism of graded rings
\begin{equation}\label{eqn:ringhomo}
R\coloneqq k[x_0,\ldots,x_{n+1}]\to L(X)
\end{equation}
induced by (\ref{eqn:ringone}) and using the commutativity of $L(X)$.

The following observation essentially proves the first
half of Theorem \ref{thm:main1}, see also Corollary \ref{cor:Thm1}.
For related results in this direction compare \cite[Prop.\ 5.14]{BFK}.

\begin{prop}\label{prop:JacHH}
If $d\leq(n+2)/2$, the ring homomorphism (\ref{eqn:ringhomo}) factors through a graded ring homomorphism from the Jacobian 
ring $J(X)=R/(\partial_if)$  to $L(X)$:
\begin{equation}\label{eqn:RJL}
\xymatrix{R(X)\ar@{->>}[r]&J(X)\ar[r]^-\pi& L(X).}
\end{equation}
\end{prop}

\begin{proof}  
This is an immediate consequence of Corollary \ref{cor:partialzero}, which claims that all partial derivatives
$\partial_if\in R_{d-1}$ of the equation $f$ defining $X$ vanish under (\ref{eqn:ringhomo}).
At this point one uses Lemma \ref{lem:Pell} to ensure that $P_{d-1}\cong j^*\ko_\Delta(d-1)$, for which the assumption on $d$ is needed.
\end{proof}

\subsection{} In this section we in addition assume that $d\mid(n+2)$ with
$d<n+2$. Then, in particular, Kuznetsov's Theorem \ref{thm:KuzCY}, Corollary \ref{cor:KuzThm}, Lemma \ref{lem:Pell}, 
and Corollary \ref{cor:Pell} all apply to our situation. It turns out that
$L(X)$ and the Jacobian ring $J(X)$ are both  Gorenstein rings of the same top
degree, cf.\ Remark \ref{rem:noaccident}.

\begin{prop}\label{prop:summ}
The ring $L(X)$ is a finite-dimensional Gorenstein ring  of degree $\sigma= (n+2)(d-2)$. More precisely, one has:
{\rm (i)} $L_0(X)\cong\HH^0(\ka_X)\cong k$;
{\rm (ii)}  $L_{d\ell}(X)=\Hom(P_0,P_{d\ell})\cong\Hom(j^*\ko_\Delta,j^*\ko_\Delta[2\ell])\cong\HH^{2\ell}(\ka_X)$,
which is of dimension one for $d\ell=\sigma$; and 
{\rm (iii)} Composition induces a non-degenerate pairing
$$ L_\ell(X)\times L_{\sigma-\ell}(X)\to L_\sigma(X)\cong k.$$
\end{prop}

\begin{proof}
Recall from Theorem  \ref{thm:KuzCY} that $\sigma/d$ is the dimension of the Calabi--Yau category $\ka_X$.

In (i) the first isomorphism is by definition and the second one follows from \cite[Cor.\ 7.5]{KuzHH}. Alternatively,
use the arguments in the proof of Lemma \ref{lem:O(1)}.

Next, (ii) follows from Corollary \ref{cor:Pell} and the definition of Hochschild cohomology, see \cite{KuzHH}.
That in (ii) the space is of dimension one for $d\ell=\sigma$ can either be deduced from Serre duality for the 
$2\sigma/d$-dimensional
Calabi--Yau category 
$\ka_X(-(n+1-d))\boxtimes\ka_X$  (cf.\ Corollary \ref{cor:KuzThm}) or from \cite[Prop.\ 5.3]{KuzCY} and \cite[Cor.\ 7.5]{KuzHH}
showing that $\HH^{2\dim\ka_X}(\ka_X)\cong\HH_{-\dim \ka_X}(\ka_X)\cong\HH_{-\dim\ka_X}(X)\cong\bigoplus H^{-\dim\ka_X+p}(X,\Omega_X^p)$,
which can be shown to be one-dimensional (using Hirzebruch's formula for $\chi(\Omega_X^p)$, see \cite{SGA},
and the Lefschetz hyperplane theorem).
Yet another possibility would be to follow the arguments in
the proof of Lemma  \ref{lem:O(1)}.

The last assertion follows again from Serre duality and the fact that $\ka_X(-(n+1-d))\boxtimes\ka_X$ is a Calabi--Yau category
of dimension $2\dim(\ka_X)=2\sigma/d$, cf.\ Corollary \ref{cor:KuzThm}.
\end{proof}

In fact, it seems likely that $\Hom(P_0,P_\ell)=0$ for all $\ell>\sigma$, which certainly is the case
if $d\mid \ell$. So presumably, $L(X)=\bigoplus_{\ell\geq0}\Hom(P_0,P_\ell)$, but this is of no importance for
what follows.

\begin{prop}\label{prop:JacHH2}
The ring homomorphism $\pi\colon J(X)\to L(X)$ in (\ref{eqn:RJL})
is injective.
%
\end{prop}

\begin{proof} For a proof of the injectivity in degree one see the proof of Lemma
\ref{lem:O(1)}. However, it can
also be seen as a consequence of the following arguments.

The graded ring homomorphism $\pi$  induces  commutative diagrams
$$\xymatrix@C=5pt{J_\ell\ar[d]_{\pi_\ell}&\times& J_{\sigma-\ell}\ar[d]_{\pi_{\sigma-\ell}}\ar[rr]&&J_\sigma\ar[d]^{\pi_\sigma}&\cong k\\
L_\ell&\times &L_{\sigma-\ell}\ar[rr]&&L_\sigma&\cong k,}$$
with both rows non-degenerate, as $J(X)$ and $L(X)$ are Gorenstein (cf.\ \cite[Thm.\ 2.5]{Don} and Proposition
\ref{prop:summ}, (iii)). Hence, the injectivity of $\pi$ 
is equivalent to the injectivity of $\pi_\sigma$ which in turn is equivalent to $\pi_\sigma\ne0$.

First we recall that $J_{d}(X)\cong H^1(X,\kt_X)$. Indeed, the normal bundle
bundle sequence $0\to \kt_X\to\kt_\PP|_X\to\ko_X(d)\to 0$ combined with the
restriction of the Euler sequence shows that 
$$J_{d}(X)\cong {\rm Coker}\left(\xymatrix{H^0(X,\ko_X(1)^{\oplus n+2})\ar[r]^-{(\partial_if)}& H^0(X,\ko_X(d))}\right)\cong H^1(X,\kt_X).$$
Moreover, the discussion in Sections \ref{sec:Pell} and \ref{sec:Atiyah} (cf.\ Lemma \ref{lem:alphabeta}) shows that
the map $$\xymatrix{H^1(X,\kt_X)\cong J_d(X)\ar[r]^-{\pi_d}& L_d(X)=\Hom(P_0,P_d)\cong\HH^2(\ka_X)}$$
can be described as the composition of the standard injection $H^1(X,\kt_X)\,\hookrightarrow \HH^2(X)$ with
the projection $ \HH^2(X)\to \HH^2(\ka_X)$, see \cite{KuzHH}. The latter is obtained by applying 
left projection: $j^*\colon \HH^2(X)= \Hom(\ko_\Delta,\ko_\Delta[2])\to\Hom(j^*\ko_\Delta,j^*\ko_\Delta[2])=\Hom(P_0,P_0[2])=\HH^2(\ka_X)$.
In particular, there exists a commutative diagram
\begin{equation}\label{eqn:JHHA}
\xymatrix{J_d^{\times\sigma/d}\ar[r]\ar[d]_{\pi_{d}^{\sigma/d}}& J_\sigma\ar[d]^{\pi_\sigma}\\
\HH^2(\ka_X)^{\times\sigma/d}\ar[r]&\HH^{2\sigma/d}(\ka_X).}
\end{equation}
On the other hand, for the Calabi--Yau category $\ka_X$ one knows by \cite[Prop.\ 5.3]{KuzCY} that
there are isomorphisms $\HH^k(\ka_X)\cong\HH_{k-\sigma/d}(\ka_X)$ compatible with the
multiplication on $\HH^*(\ka_X)$ and the $\HH^*(\ka_X)$-module structure
of $\HH_*(\ka_X)$. So, (\ref{eqn:JHHA})
can be completed by the commutative diagram
\begin{equation}\label{eqn:JHHA2}
\xymatrix{\HH^2(\ka_X)^{\times\sigma/d-1}\ar[d]^-\wr\ar@{}[r]|-{\times}&\HH^2(\ka_X)\ar[d]^-\wr\ar[r]&\HH^{2\sigma/d}(\ka_X)\ar[d]^-\wr\\
\HH^2(\ka_X)^{\times\sigma/d-1}\ar@{}[r]|-{\times}& \HH_{2-\sigma/d}(\ka_X)\ar[r]&\HH_{\sigma/d}(\ka_X)\\
\HH^2(X)^{\times \sigma/d-1}\ar[u]\ar@{}[r]|-{\times}&\HH_{2-\sigma/d}(X)\ar[r]\ar[u]^-\wr&\HH_{\sigma/d}(X)\ar[u]^-\wr,}
\end{equation}
where the lower vertical arrows between the Hochschild homology groups are indeed isomorphisms as long
as we assume $\sigma/d>2$, see \cite{KuzHH}.  To conclude, use the isomorphism between $(\HH^*(X),\HH_*(X))$
and $(HT^*(X)\coloneqq\bigoplus_{r+s=\ast} H^r(X\bigwedge^s\kt_X), H\Omega^*(X)\coloneqq\bigoplus_{q-p=\ast} H^{p,q}(X))$, see \cite{Cal},
and the fact that $H^1(X,\kt_X)^{\times\sigma/d-1}\times H\Omega^{2-\sigma/d}(X)\to H\Omega^{\sigma/d}(X)$ is non-trivial
due to a result of Griffiths, cf.\ \cite[Thm.\ 2.2]{Don}. Altogether, this proves $\pi_\sigma\ne0$.


Let us now come to the case $\sigma/d=2$, which is
the case of  cubic fourfolds and for which the proof is more direct. Here, we know that $$J_3(X)\cong H^1(X,\kt_X)\to \HH^2(\ka_X)=L_3(X)$$
is injective with image  a  subspace  of codimension two of the $22$-dimensional $\HH^2(\ka_X)$. If $\pi_6$ were trivial, then
the non-degenerate pairing $\HH^2(\ka_X)\times \HH^2(\ka_X)\to\HH^4(\ka_X)\cong k$
would be trivial on a  subspace whose dimension exceeds the maximal dimension of an isotropic
subspace. This is the contradiction which allows us to conclude that $\pi_6\ne0$ and
hence all $\pi_\ell$ are indeed injective.
\end{proof}

\subsection{} We now introduce the version of  Hochschild cohomology of $\ka_X$ of
a smooth hypersurface $X\subset\PP^{n+1}$ that is appropriate for our purpose.

\begin{definition}
The \emph{Hochschild cohomology} of the pair $(\ka_X,(1))$ is the graded subalgebra
 $$\HH^*(\ka_X,(1))\subset L(X)$$  generated by $L_1(X)$.
\end{definition}

The next result is Theorem \ref{thm:main1}.
\begin{cor}\label{cor:Thm1}
There exists a  surjection of graded rings
\begin{equation}
J(X)\twoheadrightarrow \HH^*(\ka_X,(1)),
\end{equation}
which is an isomorphism if $n+2$ divisible by $d<n+2$.
\end{cor}

\begin{proof}
The first assertion follows from the definition of $\HH^*(\ka_X,(1))$, Lemma
\ref{eqn:ringone}, and Proposition \ref{prop:JacHH}. The second part is an immediate
consequence of Proposition \ref{prop:JacHH2}.
\end{proof}

\begin{remark}\label{rem:HHnotL}
Note that the natural gradings of $\HH^*(\ka_X)$ and $\HH^*(\ka_X,(1))$ are not compatible,
as for example $\HH^{2\ell}(\ka_X)$ is mapped into $L_{d\ell}(X)$. Also observe that the projections
$\pi_\ell\colon J_\ell(X)\to L_\ell(X)$ are  in general not surjective, i.e.\
$\HH^*(\ka_X,(1))\subset L(X)$ is a proper subalgebra. For example for the cubic fourfold
we have $\dim J_3(X)=\dim H^1(X,\kt_X)=20$, whereas $\dim L_3(X)=\dim \HH^2(\ka_X)=22$.
\end{remark}

\subsection{} Consider two smooth hypersurfaces $X,X'\subset\PP^{n+1}$ of degree $1<d\leq (n+2)/2$ 
and  their associated categories $\ka_X\subset \Db(X)$ and $\ka_{X'}\subset\Db(X')$. We denote
the degree shift functors by $(1)\colon\ka_X\congpf\ka_{X'}$ and $(1)'\colon\ka_{X'}\congpf\ka_{X'}$
and their natural kernels by $P_1\in\ka_{X}(-(n+1-d))\boxtimes\ka_X$ and $P'_1\in\ka_{X'}(-(n+1-d))\boxtimes\ka_{X'}$, see
Example \ref{exa:id}, (ii).

\begin{prop}
Under the above assumption, let $\Phi=\bar\Phi_P\colon\ka_X\congpf\ka_{X'}$ be a Fourier--Mukai equi\-valence.
Then an isomorphism of the natural Fourier--Mukai kernels of  $\Phi\circ (1)$ and $(1)'\circ\Phi$,
\begin{equation}\label{eqn:FMdeg}
P_1\circ P\cong P\circ P_1',
\end{equation}
induces an isomorphism of graded algebras 
$$\HH^*(\ka_X,(1))\cong\HH^*(\ka_{X'},(1)').$$
\end{prop}

\begin{proof}
By definition of a Fourier--Mukai functor, 
$P\in\ka_{X}(-(n+1-d))\boxtimes\ka_{X'}$, see Remark \ref{rem:FMequiCY}.  Hence, both sides of (\ref{eqn:FMdeg})
are objects in $\ka_X(-(n+1-d))\boxtimes\ka_{X'}$. Moreover, successive convolution with $P_1$ from the left yields isomorphisms
$P_\ell\circ P\cong P\circ P_\ell'$ for all $\ell\geq0$. Alternatively, $P_\ell\circ P$ can be seen as the image 
of $P_\ell$ under the equivalence $(~~\circ P)\cong{\rm id}\boxtimes \Phi\colon\ka_X(-n+1-d))\boxtimes\ka_{X}\congpf \ka_X(-n+1-d))\boxtimes\ka_{X'}$
and $P\circ P_\ell'$ as the image of $P_\ell'$ under the equivalence
$(P\circ~~)\cong\Psi\boxtimes {\rm id}\colon\ka_{X'}(-(n+1-d))\boxtimes\ka_{X'}\congpf\ka_X(-(n+1-d))\boxtimes\ka_{X'}$.
Here, $\Psi\colon\ka_{X'}(-(n+1-d))\congpf\ka_X(-(n+1-d))$ is the Fourier--Mukai equivalence
with kernel in $\ka_{X'}\boxtimes\ka_X(-(n+1-d))$ given by  applying the transposition to $P$.
The arguments in the geometric case can be easily adapted to show that $\Psi$ is indeed an equivalence.

Hence, the equivalence $(\Psi\boxtimes {\rm id})^{-1}\circ ({\rm id}\boxtimes\Phi)\colon \ka_X(-(n+1-d))\boxtimes\ka_{X}\congpf \ka_{X'}(-(n+1-d))\boxtimes\ka_{X'}$ sends $P_\ell$ to $P_\ell'$ and, therefore, defines isomorphisms $L_\ell(X)\cong L_\ell(X')$, $\ell\geq0$, 
compatible with composition. Restricted to the sub-algebras generated by $L_1$, this yields the desired isomorphism
of graded algebras $\HH^*(\ka_X,(1))\congpf\HH^*(\ka_{X'},(1)')$.
\end{proof}

As it is expected that  Fourier--Mukai kernels $P\in \ka_{X}(-(n+1-d))\boxtimes\ka_{X'}$
of  Fourier--Mukai equivalences $\Phi=\bar\Phi_P\colon \ka_X\congpf\ka_{X'}$ are unique, an isomorphism
(\ref{eqn:FMdeg}) should exist whenever $\Phi\circ(1)\cong(1)'\circ\Phi$.
This is certainly the case when $\ka_X\cong\Db(S,\alpha)$ for a twisted K3 surface $(S,\alpha)$ due to \cite{CanSte,OrlovK3}.

\begin{cor}\label{cor:Jaciso} 
Let $X,X'\subset \PP^{n+2}$ be smooth hypersurfaces of degree  $1<d\leq(n+2)/2$ with $d\mid(n+2)$.
Then an isomorphism of the natural Fourier--Mukai kernels of  $\Phi\circ (1)$ and $(1)'\circ\Phi$,
\begin{equation}\label{eqn:FMdeg}
P_1\circ P\cong P\circ P_1',
\end{equation}
induces an isomorphism of graded algebras $J(X)\cong J(X')$
and, therefore, an isomorphism $$X\cong X'.$$
\end{cor}

\begin{proof}
The isomorphism between the Jacobian rings follows from the above proposition and Corollary \ref{cor:Thm1}. That
the isomorphism between the Jacobian ring implies  the existence of an isomorphism $X\cong X'$
is an immediate consequence of the Mather--Yau theorem, see \cite[Prop.\ 1.1]{Don} or \cite[Lem.\ 18.31]{VoisinHodge}.
It may be worth noting that the isomorphism $J(X)\cong J(X')$ itself may not lift directly to an automorphism of
$k[x_0,\ldots,x_{n+1}]$ identifying the equations of $X$ and $X'$.
\end{proof}

\section{Hodge theory}
The existence of a Hodge isometry $H^4(X,\ZZ)_{\rm pr}\cong H^4(X',\ZZ)_{\rm pr}$ is shown to
yield a Hodge isometry $\widetilde H(\ka_X,\ZZ)\cong\widetilde H(\ka_{X'},\ZZ)$ with additional properties.
It will subsequently be lifted to an equivalence $\ka_X\cong\ka_{X'}$ that on the level of Hochschild cohomology
yields an isomorphism between the Jacobian rings. From now on, we work over $\CC$.

\subsection{} Let us briefly recall the Hodge structure of weight two
$\widetilde H(\ka_X,\ZZ)$ of the K3 category $\ka_X\subset{\rm D}^{\rm b}(X)$ associated 
with a smooth cubic $X\subset\PP^5$ as defined by Addington and Thomas in  \cite{AT}. 
As a lattice, $\widetilde H(\ka_X,\ZZ)$ is the orthogonal
complement of $\langle[\ko_X],[\ko_X(1)],[\ko_X(2)]\rangle\subset K_{\rm top}(X)$ with the quadratic form given by
 the Mukai pairing. The Hodge structure is determined by $\widetilde H^{2,0}(\ka_X)$  defined as the pull-back of $H^{3,1}(X)$ via the Mukai vector.
Furthermore, there exists a natural primitive inclusion (see \cite[Prop.\ 2.3]{AT}): $$\iota_X\colon H^4(X,\ZZ)_{\rm pr}\hookrightarrow\widetilde H(\ka_X,\ZZ),$$ which is compatible up to sign with the intersection form on $H^4(X,\ZZ)_{\rm pr}$ and the Mukai pairing on $\widetilde H(\ka_X,\ZZ)$. Moreover, $\iota_X$ respects the Hodge structures (up to Tate twist), i.e.\ it  restricts to an isomorphism $H^{3,1}(X)\congpf\widetilde H^{2,0}(\ka_X,\ZZ)$. The ortho\-gonal complement of the inclusion is the primitive sublattice spanned by $\lambda_j=[i^*\ko(j)]$, $j=1,2$,
where as before,
 $i^*\colon{\rm D}^{\rm b}(X)\to\ka_X$ is the left adjoint of the natural
 inclusion $i_*\colon\ka_X\,\hookrightarrow{\rm D}^{\rm b}(X)$. The choice of the generators induces an isometry ${\rm Im}(\iota_X)^\perp\cong A_2$,
which we shall tacitly fix throughout. The induced inclusion
$$H^4(X,\ZZ)_{\rm pr}\oplus A_2\,\hookrightarrow \widetilde H(\ka_X,\ZZ)$$ is
of index three, cf.\ \cite[Sect.\ 14.0.2]{HuyK3}, and its quotient $H_X\coloneqq \widetilde H(\ka_X,\ZZ)/(H^4(X,\ZZ)_{\rm pr}\oplus A_2)$ can be viewed naturally as a subgroup of the discriminant group $$A_{H^4(X,\ZZ)_{\rm pr}\oplus A_2}\cong A_{H^4(X,\ZZ)_{\rm pr}}\oplus A_{A_2}\cong \ZZ/3\ZZ\oplus\ZZ/3\ZZ.$$ The two projections from $H_X$ yield an isomorphism 
\begin{equation}\label{eqn:gamma}
\gamma_X\colon A_{H^4(X,\ZZ)_{\rm pr}}\congpf A_{A_2}~ (\cong\ZZ/3\ZZ).
\end{equation}

Also note that $H^4(X,\ZZ)_{\rm pr}$ with the reversed sign has two positive directions which are naturally oriented by taking real
and imaginary part of any generator of $H^{3,1}(X)$ (or, equivalently, of $\widetilde H^{2,0}(\ka_X)$ via $\iota_X$). Clearly, by picking the base $\lambda_1,\lambda_2$ of the positive
definite orthogonal complement ${\rm Im}(\iota_X)^\perp\cong A_2$ one also gives this part a natural orientation. Put together,
the four positive directions of $\widetilde H(\ka_X,\ZZ)$ are endowed with a natural orientation.

The orthogonal group  of $A_2$ is known to be $\OO(A_2)\cong {\mathfrak S}_3\times\ZZ/2\ZZ$. Here,
the second factor acts by a global sign change, whereas the first one is the Weyl group acting by permutation
of the unit vectors $e_i$, where $A_2\,\hookrightarrow \RR^3$ via $\lambda_1\mapsto e_1-e_2$ and $\lambda_2\mapsto e_2-e_3$.
 Note that ${\mathfrak S}_3$  can also be described as the kernel of the restriction $\OO(A_2)\twoheadrightarrow\OO(A_{A_2})\cong\OO(\ZZ/3\ZZ)\cong\ZZ/2\ZZ$.
Moreover,  $g\in \OO(A_2)$ preserves the natural orientation of the lattice $A_2$ if and only
if it is contained in ${\mathfrak A}_3\times\ZZ/2\ZZ$, which still surjects onto $\OO(A_{A_2})$. 
A generator of ${\mathfrak A}_3$ is described by the cyclic permutation of $\lambda_1,\lambda_2,-\lambda_1-\lambda_2$,
cf.\ \cite[Rem.\ 2.1]{HuyCubic}.

\begin{remark}
The category $\ka_X$ is equipped with the natural auto-equivalence $(1)\colon\ka_X\congpf\ka_X$ which is  of Fourier--Mukai type, see 
Example \ref{exa:id}, (ii). The induced action 
\begin{equation}\label{eqn:degcoh}
(1)^H\colon\widetilde H(\ka_X,\ZZ)\congpf\widetilde H(\ka_X,\ZZ)
\end{equation}
is the identity on $H^4(X,\ZZ)_{\rm pr}$ and  cyclicly permutes $\lambda_1,\lambda_2,-\lambda_1-\lambda_2$
in the orthogonal complement $A_2$, cf. \cite[Prop.\ 3.12]{HuyCubic}.
\end{remark}
\subsection{} Any Hodge isometry $\varphi\colon H^4(X,\ZZ)\congpf H^4(X',\ZZ)$ mapping $h_X$ to $h_{X'}$ induces a Hodge iso\-metry
$\varphi\colon H^4(X,\ZZ)_{\rm pr}\congpf H^4(X',\ZZ)_{\rm pr}$. 
Conversely, adapting \cite[Cor.\ 1.5.2]{Nikulin} to the present context, one  shows that
any Hodge iso\-metry $H^4(X,\ZZ)_{\rm pr}\congpf H^4(X',\ZZ)_{\rm pr}$ extends to a Hodge isometry
$H^4(X,\ZZ)\congpf H^4(X',\ZZ)$ mapping $h_X$ to $\pm h_{X'}$. (Use the natural isomorphism between the discriminant groups
$A_{H^4(X,\ZZ)_{\rm pr}}\cong A_{A_2}\cong\ZZ/3\ZZ$.)

For the following, both primitive cohomologies are considered with their
natural inclusions $\iota_X\colon H^4(X,\ZZ)_{\rm pr}\hookrightarrow \widetilde H(\ka_X,\ZZ)$ and $\iota_{X'}\colon
H^4(X',\ZZ)_{\rm pr}\hookrightarrow \widetilde H(\ka_{X'},\ZZ)$,
respectively.

\begin{prop}\label{prop:extend}
Any Hodge isometry $\varphi\colon H^4(X,\ZZ)_{\rm pr}\congpf H^4(X',\ZZ)_{\rm pr}$
extends to an orientation preserving Hodge isometry
$\tilde\varphi \colon\widetilde H(\ka_X,\ZZ)\congpf \widetilde H(\ka_{X'},\ZZ)$ that commutes with $(1)^H$ in (\ref{eqn:degcoh}).
\end{prop}

\begin{proof}
The key is Nikulin's classical result \cite[Cor.\ 1.5.2]{Nikulin}, see also \cite[Ch.\ 14]{HuyK3}, showing that
$\varphi$ can be extended to an isometry that restricts to a given $g\in\OO(A_2)$ between the orthogonal
complements (which we have identified with $A_2$) if and only if $\bar g\circ \gamma_X=\gamma_{X'}\circ \bar\varphi\colon A_{H^4(X,\ZZ)_{\rm pr}}\to A_{A_2}$. Here, $\bar g$ and $\bar\varphi$ are the induced maps between the discriminant groups and $\gamma_X,\gamma_{X'}$ are as
in (\ref{eqn:gamma}). 

Thus, any lift $g\in \OO(A_2)$
of $\gamma_{X'}\circ \bar\varphi\circ\gamma_X^{-1}\in \OO(A_{A_2})$ defines an extension of the Hodge isometry $\varphi$
to an isometry $\tilde\varphi\colonÊ\widetilde H(\ka_X,\ZZ)\congpf\widetilde H(\ka_{X'},\ZZ)$ which is then automatically compatible
with the Hodge structures.
Moreover, there always exists a lift $g\in{\mathfrak A}_3\times\ZZ/2\ZZ$ of $\gamma_{X'}\circ \bar\varphi\circ\gamma_X^{-1}$ such that $\tilde\varphi$ is not only
a Hodge isometry but also preserves the natural orientation of the four positive directions. Any such $g$ automatically commutes with the cyclic permutation of $\lambda_1,\lambda_2,-\lambda_1-\lambda_2$ and, therefore, the induced $\tilde\varphi$ commutes with $(1)^H$.
\end{proof}

\section{Deformation theory}
In this section we show how to combine Corollary \ref{cor:Jaciso} with results from \cite{HuyCubic} to prove the global Torelli theorem
first for general and then for all cubics.

\subsection{}
We  start by recalling \cite[Thm.\ 1.2]{HuyCubic}, which describes the group of auto-equivalences of $\ka_X$ for a very general cubic.
Here, a cubic $X\in|\ko(3)|$ is very general if it is contained in the complement of a countable union of proper closed subsets of $|\ko(3)|$.
The proof in \cite{HuyCubic} does not depend on the global Torelli theorem
for cubic fourfolds, but makes use of results on ${\rm Aut}({\rm D}^{b}(S,\alpha))$ for twisted K3 surfaces $(S,\alpha)$ without
spherical objects, see \cite{HMS2}.

\begin{prop}\label{prop:Autgen}
Let $X\subset\PP^5$ be a very general smooth cubic. Then the group $\Aut(\ka_X)$ of
Fourier--Mukai auto-equivalences $\Phi\colon\ka_X\congpf\ka_X$ (see Definition \ref{def:FM}) is an infinite cyclic group containing $\ZZ\cdot [1]$ as an index three subgroup.\qed 
\end{prop}

The degree shift functor $(1)\in\Aut(\ka_X)$ is symplectic, i.e.\ it acts trivially on the transcendental part (which for
a general cubic is $H^4(X,\ZZ)_{\rm pr}\cong{\rm Im}(\iota_X)$), and generates the quotient $\Aut_s(\ka_X)/\ZZ\cdot[2]$.

\begin{cor}\label{cor:twistcom}
Let $X,X'\subset\PP^5$ be two very general cubics and let $\Phi=\bar\Phi_P\colon\ka_X\congpf\ka_{X'}$ be a Fourier--Mukai
equivalence for which the induced action $\Phi^H\colon \widetilde H(\ka_X,\ZZ)\congpf\widetilde H(\ka_{X'},\ZZ)$ commutes with the action
of the degree shift $(1)^H$ on $\widetilde H(\ka,\ZZ)$.
Then $\Phi$ commutes  with the degree shift functor on both sides, i.e.\ $\Phi\circ (1)\cong(1)'\circ\Phi$.
\end{cor}

\begin{proof}
Indeed, under the assumption $\Psi\coloneqq\Phi^{-1}\circ(1)'\circ\Phi\circ(-1)\in\Aut(\ka_X)$ acts trivially on $\widetilde H^{1,1}(\ka_X,\ZZ)\cong A_2$
and hence $\Psi\cong [k]$. However, as $[2]\cong(3)$ on both sides, the relation $(1)'\circ\Phi\cong \Phi\circ(1)\circ[k]$ automatically implies
$k=0$. 
\end{proof}

If Fourier--Mukai kernels of Fourier--Mukai equivalences $\ka_X\congpf\ka_{X'}$ are unique, then $\Phi\circ(1)\cong(1)'\circ\Phi$ would immediately yield (\ref{eqn:FMdeg}) and hence $X\cong X'$ by Corollary \ref{cor:Jaciso}.

\begin{remark}
In fact, as was stated in \cite[Thm.\ 1.5]{HuyCubic}, two very general cubics $X$ and $X'$ are isomorphic if and only if there exists
a Fourier--Mukai equivalence $\ka_X\cong\ka_{X'}$. However, the proof given in \cite{HuyCubic} relies on
the global Torelli theorem. Indeed, if $X,X'$ are very general,
then $\widetilde H^{1,1}(\ka_X,\ZZ)\cong\widetilde H^{1,1}(\ka_{X'},\ZZ)$ is just an isometry of $A_2$
and the Hodge isometry of its orthogonal complement can therefore be read as a Hodge isometry
$H^4(X,\ZZ)_{\rm pr}\cong H^4(X',\ZZ)_{\rm pr}$.
\end{remark}

\begin{remark}\label{rem:commtwisted}
Note that all that was needed in the proof of Corollary \ref{cor:twistcom} was that the kernel
of $\Aut_s(\ka_X)\to\Aut(\widetilde H(\ka_X,\ZZ))$ is $\ZZ\cdot[2]$. This is true for very general
cubics, but also for cubics $X$ for which there exists a twisted K3 surface $(S,\alpha)$ with $\ka_X\cong\Db(S,\alpha)$ and
without any $(-2)$-classes in $N(S,\alpha)\cong \widetilde H^{1,1}(\ka_X,\ZZ)$, see \cite{HMS2}. Moreover, in the latter case $\Phi\circ(1)\cong (1)'\circ\Phi$ implies $P_\ell\circ P\cong P\circ P_\ell'$, $\ell\geq1$, due to the uniqueness of Fourier--Mukai kernel for the category of twisted coherent 
sheaves \cite{CanSte}.
\end{remark}

\subsection{Proof of Theorem \ref{thm:GT}} 
Let $\varphi\colon H^4(X,\ZZ)_{\rm pr}\congpf H^4(X',\ZZ)_{\rm pr}$ be a Hodge isometry. According to
Proposition \ref{prop:extend}, $\varphi$ can be extended to an orientation preserving Hodge isometry
$\tilde\varphi\colon \widetilde H(\ka_X,\ZZ)\congpf\widetilde H(\ka_{X'},\ZZ)$ that commutes with the action of the degree shift $(1)^H$
on the two sides, which corresponds to a cyclic permutation of $\lambda_1,\lambda_2,-\lambda_1-\lambda_2$.

Due to the local Torelli theorem, $\varphi$ can be naturally extended to Hodge isometries
\begin{equation}\label{eqn:parallel}
\varphi_t\colon H^4(X_t,\ZZ)_{\rm pr}\congpf H^4(X'_t,\ZZ)_{\rm pr}
\end{equation}
for all local deformations $X_t$ and $X'_t$. More precisely, there exists an identification
${\rm Def}(X)\cong {\rm Def}(X')$ between the bases of the universal deformation spaces (think of them as open sets
of the period domain) such that parallel transport
induces (\ref{eqn:parallel}). Simultaneously, the $\varphi_t$ can be extended to orientation
preserving Hodge isometries $\tilde\varphi_t\colon \widetilde H(\ka_{X_t},\ZZ)\congpf\widetilde H(\ka_{X'_t},\ZZ)$
commuting with the degree shift.

The set $D'$ of points $t\in{\rm Def}(X)\cong{\rm Def}(X')$ for which there exists a twisted K3 surface $(S_t,\alpha_t)$
without spherical objects and
orientation preserving Hodge isometries
\begin{equation}\label{eqn:cubicK3}
\tilde\varphi_t\colon\widetilde H(\ka_{X_t},\ZZ)\congpf \widetilde H(S_t,\alpha_t,\ZZ)\congpf \widetilde H(\ka_{X'_t},\ZZ)
\end{equation}
is dense (more precisely, a countable union of hypersurfaces, which using their Hodge theoretic description can be
seen to be also analytically dense), see \cite[Cor.\ 2.16]{HuyCubic}. Moreover, due to \cite[Thm.\ 1.4]{HuyCubic} which is the twisted version
of \cite[Thm.\ 1.1]{AT}, we may assume that  the Hodge isometries (\ref{eqn:cubicK3}) can be lifted to  Fourier--Mukai equivalences 
\begin{equation}
\Phi_t=\bar\Phi_{P_t}\colon\ka_{X_t}\congpf{\rm D}^{\rm b}(S_t.\alpha_t)\congpf\ka_{X'_t},
\end{equation} 
where $P_t\in\ka_{X_t}(-2)\boxtimes\ka_{X_t'}$.
Note that at this point one implicitly uses the derived global Torelli theorem for K3 surfaces, 
which ultimately relies on the classical global Torelli theorem for K3 surfaces.
For example, it is used in \cite[Prop.\ 5.1]{AT} and, more generally,
in the description of the image of ${\rm Aut}(\Db(S,\alpha))\to {\rm Aut}(\widetilde H(S,\alpha,\ZZ))$.

Remark \ref{rem:commtwisted} applies to $t\in D'$  and yields $P_{\ell t}\circ P_t\cong P_t\circ P'_{\ell t}$, where
$P_{\ell t}, P'_{\ell t}$ are the natural kernels for the degree shift functor $(\ell)$ on $\ka_{X_t}$ and $\ka_{X'_t}$, respectively.
Corollary \ref{cor:Jaciso} then yields $X_t\cong X'_t$ for all $t$ in the dense set $D'$  and, as the moduli space of cubics is separated, this shows that
$X\cong X'$. This concludes the proof of Theorem \ref{thm:GT}.\qed



\begin{remark}
Note that the arguments in \cite{AT,HuyCubic} rely on the fact that all cubics $X\in\kc_8$ contain a plane. This is a result due to Voisin, see
 \cite[Sect.\ 3]{VoisinGT}. It is used to prove the global Torelli theorem first for these cubics before extending it then to all.
\end{remark}


\section{Further comments}\label{sec:MF}

Although matrix factorizations have not been used in any of the proofs above, some of the arguments
are clearly motivated by thinking in terms of ${\rm MF}(f,\ZZ)$. We briefly recall the interpretation of $\ka_X$ as
the category of graded matrix factorizations and explain how to view some of the techniques in this light.

 The  category $\ka_X$ of a smooth hypersurface $X\subset\PP^{n+1}$ can also be described as a category
of graded matrix  factorizations. More precisely, there exists an exact linear equivalence
$${\rm MF}(f,\ZZ)\cong\ka_X,$$
where $f\in R_d= k[x_0,\ldots,x_{n+1}]_d$ is defining $X$.
In fact, Orlov constructs a series of fully faithful embeddings
$\Phi_i\colon{\rm MF}(f,\ZZ)\,\hookrightarrow\Db(X)$, see \cite[Thm.\ 2.5]{Orlov}, and for $i=1$ the image
is the subcategory $\ka_X=\langle\ko,\ldots,\ko(n+1-d)\rangle^\perp={}^\perp\langle\ko(d-n-2),\ldots,\ko(-1)\rangle$.
 
 The
objects of ${\rm MF}(f,\ZZ)$ are pairs $(K\stackrel{\alpha}{\to}L\stackrel{\beta}{\to} K(d))$,
where $K$ and $L$ are finitely generated, free, graded $R$-modules and $\alpha,\beta$ are
graded $R$-module homomorphisms with $\beta\circ\alpha=f\cdot{\rm id}=\alpha(d)\circ\beta$. 
Recall that $K(\ell)$ for a graded $R$-module $K=\bigoplus K_i$ is the graded module
with $K(\ell)_i=K_{\ell+i}$. Morphisms in ${\rm MF}(f,\ZZ)$ consist of homotopy classes of pairs of
graded homomorphisms $g\colon K\to K'$ and $h\colon L\to L'$  with $\alpha'\circ g=h\circ\alpha$ and
$\beta'\circ h= g(d)\circ\beta$ and the shift functor making ${\rm MF}(f,\ZZ)$ into a triangulated category
is given by $$(K\stackrel{\alpha}{\to}L\stackrel{\beta}{\to} K(d))[1]\coloneqq(L\stackrel{-\beta}{\to}K(d)\stackrel{-\alpha}{\to} L(d)).$$

The degree shift functor for ${\rm MF}(f,\ZZ)$ is by definition the auto-equivalence given by
\begin{eqnarray*}
(1)\colon {\rm MF}(f,\ZZ)&\congpf& {\rm MF}(f,\ZZ)\\
 (K\stackrel{\alpha}{\to}L\stackrel{\beta}{\to} K(d))&\mapsto&(K(1)\stackrel{\alpha(1)}{\to}L(1)\stackrel{\beta(1)}{\to} K(d+1)), 
 \end{eqnarray*}
which obviously satisfies $(d)\cong [2]$, cf.\ Corollary \ref{cor:Pell}. 

\begin{remark}\label{rem:einsgltwist}
According to \cite[Prop.\ 5.8]{BFK} under the fixed equivalence ${\rm MF}(f,\ZZ)\cong\ka_X$ (which corresponds
to $\Phi_1$ in \cite[Thm.\ 2.5]{Orlov}) the degree shift functor $(1)$  on ${\rm MF}(f,\ZZ)$
is isomorphic to the auto-equivalence $i^*\circ(\ko(1)\otimes(~))$ of $\ka_X$,
which is the degree shift functor $(1)$ on $\ka_X$ as introduced in Example \ref{exa:id}, (ii).

Of course, once this
is established, the isomorphism $(d)\cong[2]$ in Corollary \ref{cor:Pell} is immediate. However, only on the level of functors, but not on the level
of Fourier--Mukai kernels, which was crucial for our purposes.
\end{remark}

Consider the image of a section $s\in H^0(X,\ko_X(1))$  under the isomorphism (\ref{eqn:ringone}) and view it as
a morphism between Fourier--Mukai kernels $s\colon P_0=j^*\ko_\Delta\to P_1= j^*\ko_\Delta(1)$. The induced natural
transform $s\colon {\rm id}\to (1)$ between the auto-equivalences ${\rm id},(1)\colon\ka_X\cong {\rm MF}(f,\ZZ)\congpf\ka_X\cong{\rm MF}(f,\ZZ)$ can be described in terms of matrix factorizations as given by multiplication by
the section $s$, i.e.
\begin{equation}\label{eqn:sss}
(\xymatrix{K\ar[r]^-{\alpha}& L\ar[r]^-{\beta}& K(d)})\xymatrix{\ar[r]^{(s,s,s)}&} (\xymatrix{K(1)\ar[r]^-{\alpha(1)}& L(1)\ar[r]^-{\beta(1)}&K(d+1)}).
\end{equation} A similar description
holds for the transforms ${\rm id}\to(\ell)$ induced by sections $s\in H^0(X,\ko_X(\ell))$.

\begin{remark}
The naive idea behind Proposition \ref{prop:JacHH} is the chain rule. Indeed, for $s=\partial_if\in H^0(X,\ko_X(d-1))$
the map (\ref{eqn:sss})
is homotopic to zero; just use the chain rule $\partial_if=\partial_i (\beta\circ \alpha)=
\partial_i\alpha\circ\beta+\alpha\circ\partial_i\beta$ to see that $\partial_i\beta\colon L\to K(d-1)$
and $\partial_i\alpha\colon K\to L(-1)$ define a homotopy $\partial_if\sim 0$, see for example \cite{Dyck}.
However, Proposition  \ref{prop:JacHH} cannot be replaced by this easy observation, as the natural map 
$\Hom(P_0,P_\ell)\to {\rm Fun}({\rm id},(\ell))$
is not always injective.\footnote{Indeed, it can be shown that ${\rm Fun}({\rm id},[i])=0$ for the bounded derived category
$\Db(Y)$ of any smooth projective variety $Y$ and for $i>\dim (Y)$.
On the other hand, $\Hom(P_0,P_\sigma)\ne 0$. So, e.g.\ for cubic fourfolds with $\ka_X\cong\Db(S)$,
the map $\Hom(P_0,P_6)\to{\rm Fun}({\rm id},(6))\cong{\rm Fun}({\rm id},[4])=0$ is not injective.}
\end{remark}

It seems feasible that some of the arguments in this paper can indeed be rephrased in the language
of matrix factorizations. However, the transition between the two points of view is often involved,
cf.\ \cite{BFK}. Already identifying the two degree shift functors $(1)$ or verifying that $s\in H^0(X,\ko_X(1))$ really
yields
(\ref{eqn:sss}) are non-trivial matters.
In the end, we decided that
staying in the derived context throughout makes the arguments cleaner.

Passing from the formalism of matrix factorizations to the equivalent formalism of general factorizations
(see for instance \cite{ADS,BFK2}), we may represent the functors $(i)\colon {\rm MF}(f,\ZZ)\to {\rm MF}(f,\ZZ)$
by a natural Fourier--Mukai kernel $Q_i$ in (the appropriately graded version of) 
${\rm MF}(-f\boxplus  f)$. We may then define $L_{\rm MF}(X)\coloneqq \bigoplus \Hom(Q_0,Q_i)$.
A more general version of this ring is computed in \cite{BFK2}. Their methods can presumably
be used show that $J(X)$ is isomorphic to the subring of $L_{\rm MF}(X)$ generated by its degree one part.
It is natural to conjecture that $L_{\rm MF}(X ) \cong  L(X )$, but for the reasons listed, proving this is a non-trivial task.



\end{document}